\documentclass[11pt,twoside]{preprint}

\usepackage[full]{textcomp}
\usepackage{times}

\usepackage[cal=boondoxo]{mathalfa}
\usepackage{microtype}

\usepackage{amsmath}
\usepackage{amsfonts}
\usepackage{amssymb}
\usepackage[full]{textcomp}
\usepackage[osf]{newtxtext}
\usepackage[all]{xy}
\usepackage[utf8]{inputenc} 
\usepackage{comment}

\usepackage{hyperref}
\usepackage{breakurl}

\usepackage{mathrsfs}
\usepackage{enumitem}

\usepackage{tikz}
\usepackage{dsfont}
\usepackage{mhequ} 
\usepackage{mhsymb} 
\usepackage{mhenvs} 
\usepackage{array}

\usepackage{wasysym}
\usepackage{centernot}
\usepackage{mathtools}

\setlist{noitemsep,nolistsep,leftmargin=1.7em}
\usetikzlibrary{snakes}
\usetikzlibrary{decorations}
\usetikzlibrary{positioning}
\usetikzlibrary{shapes}

\DeclareFontFamily{U}{mathx}{\hyphenchar\font45}
\DeclareFontShape{U}{mathx}{m}{n}{
      <5> <6> <7> <8> <9> <10>
      <10.95> <12> <14.4> <17.28> <20.74> <24.88>
      mathx10
      }{}
\DeclareSymbolFont{mathx}{U}{mathx}{m}{n}
\DeclareMathSymbol{\bigtimes}{1}{mathx}{"91}

\def\s{\mathfrak{s}}

\def\emptyset{{\centernot\ocircle}}

\definecolor{darkred}{rgb}{0.7,0.1,0.1}
\definecolor{darkblue}{rgb}{0.1,0.1,0.8}
\definecolor{darkgreen}{rgb}{0.1,0.7,0.1}
\def\restr{\mathord{\upharpoonright}}

\providecommand{\figures}{false}
{ \ifthenelse{\equal{\figures}{false}} {#1}{\[ {\rm Figure \ missing !} \]} }{}
\def\id{\mathrm{id}}

\def\Labo{\mathfrak{o}}

\def\CH{\mathcal{H}}

\def\CG{\mathcal{G}}
\def\CJ{\mathcal{J}}
\def\CA{\mathcal{A}}

\def\CM{\mathcal{M}}
\def\CT{\mathcal{T}}
\def\RR{\mathfrak{R}}

\def\Labe{\mathfrak{e}}

\def\Labn{\mathfrak{n}}

\def\Labhom{\mathfrak{t}}
\def\Lab{\mathfrak{L}}
\def\Adm{\mathfrak{A}}

\def\Deltam{\Delta^{\!-}}

\def\Deltap{\Delta^{\!+}}

\def\${|\!|\!|}

\def\Co{\mathscr{C}}
\def\DeltaB{\Delta^{\!M^{\circ}}}
\def\DeltaM{\Delta^{\!M}}

\def\DeltaQ{\hat \Delta_{1}}

\newcommand{\mfL}{\mathfrak{L}}

\newcommand{\mfl}{\mathfrak{l}}

\def\Trees{\mathfrak{T}}

\def\Forests{\mathfrak{F}}

\def\Tra{\mathfrak{F}}

\def\CR{\mathcal{R}}
\def\CF{\mathcal{F}}

\newenvironment{DIFnomarkup}{}{} 

\newcommand{\I}{{\mathcal I}}

\newfont{\indic}{bbmss12}

\def\un#1{\hbox{{\indic 1}$_{#1}$}}

\def\Symb{\mathrm{Symb}}

\def\PPi{\boldsymbol{\Pi}}

\colorlet{symbols}{blue!90!black}
\colorlet{testcolor}{green!60!black}
\colorlet{connection}{red!30!black}
\def\symbol#1{\textcolor{symbols}{#1}}
\def\symbol#1{\textcolor{symbols}{#1}}

\tikzset{
root/.style={circle,fill=black!50,inner sep=0pt, minimum size=3mm},
        circ/.style={circle,fill=white,draw=black,very thin,inner sep=.5pt, minimum size=1.2mm},
        dot/.style={circle,fill=black,inner sep=0pt, minimum size=1.2mm},
        dotred/.style={circle,fill=black!50,inner sep=0pt, minimum size=2mm},
        var/.style={circle,fill=black!10,draw=black,inner sep=0pt, minimum size=3mm},
        kernel/.style={semithick,shorten >=2pt,shorten <=2pt},
        kernel1/.style={thick},
        kernels/.style={snake=zigzag,shorten >=2pt,shorten <=2pt,segment amplitude=1pt,segment length=4pt,line before snake=2pt,line after snake=5pt,},
		kernels1/.style={snake=zigzag,segment amplitude=0.5pt,segment length=2pt},
		rho1/.style={densely dotted,semithick},
        rho/.style={densely dashed,semithick,shorten >=2pt,shorten <=2pt},
           testfcn/.style={dotted,semithick,shorten >=2pt,shorten <=2pt},
        renorm/.style={shape=circle,fill=white,inner sep=1pt},
        labl/.style={shape=rectangle,fill=white,inner sep=1pt},
        xic/.style={very thin,circle,fill=symbols,draw=black,inner sep=0pt,minimum size=1.2mm},
        xi/.style={very thin,circle,fill=blue!10,draw=black,inner sep=0pt,minimum size=1.2mm},
	xib/.style={very thin,circle,fill=blue!10,draw=black,inner sep=0pt,minimum size=1.6mm},
	xie/.style={very thin,circle,fill=green!50!black,draw=black,inner sep=0pt,minimum size=1mm},
	xid/.style={very thin,circle,fill=symbols,draw=black,inner sep=0pt,minimum size=1.6mm},
	edgetype/.style={very thin,circle,draw=black,inner sep=0pt,minimum size=5mm},
	nodetype/.style={very thick,circle,draw=black,inner sep=0pt,minimum size=5mm},
	kernels2/.style={very thick,draw=connection,segment length=12pt},
clean/.style={thin,circle,fill=black,inner sep=0pt,minimum size=1mm},	not/.style={thin,circle,fill=symbols,draw=connection,fill=connection,inner sep=0pt,minimum size=0.8mm},
	>=stealth,
        }

\tikzset{ individus/.style={scale=0.40,draw,circle,thick,fill=black!10},
 individu/.style={scale=0.40,draw,circle,thick,fill=black!50},       } 
\makeatletter
\def\DeclareSymbol#1#2#3{\expandafter\gdef\csname MH@symb@#1\endcsname{\tikz[baseline=#2,scale=0.15,draw=symbols]{#3}}\expandafter\gdef\csname MH@symb@#1s\endcsname{\scalebox{0.7}{\tikz[baseline=#2,scale=0.15,draw=symbols]{#3}}}}
\def\<#1>{\csname MH@symb@#1\endcsname}
\makeatother

 \def\1{\mathbf{\symbol{1}}}

 \def\un#1{\hbox{{\indic 1}$_{#1}$}}
\def\one{\mathbf{1}}

\DeclareMathAlphabet{\mathpzc}{OT1}{pzc}{m}{it}

\def\un#1{\hbox{{\indic 1}$_{#1}$}}
\def\one{\mathbf{1}}

\def\Deltap{\Delta^{\!+}}

\def\Deltam{\Delta^{\!-}}

\def\PPi{\boldsymbol{\Pi}}

\def\id{\mathrm{id}}

\def\simnot{\stackrel{\vbox to 0.15em{\hbox{\kern0.07em$^\circ$}}}{\sim}}

\begin{document}

\title{Recursive formulae in regularity structures}
\author{Y. Bruned}
\institute{Imperial College London \\
\email{y.bruned@imperial.ac.uk} 
}
\maketitle

\begin{abstract} We construct  renormalised models of regularity structures by using a recursive formulation for the structure group and for the renormalisation group. This construction covers all the examples of singular SPDEs which have been treated so far with the theory of regularity structures and improves the renormalisation procedure based on Hopf algebras given in Bruned-Hairer-Zambotti (2016).
\end{abstract}

\setcounter{tocdepth}{1}
\tableofcontents

\section{Introduction}

During the last years, the theory of regularity structures introduced by Martin Hairer in \cite{reg} has proven to be an essential tool for solving singular SPDEs of the form:
\begin{equs}\label{standard equation}
(\partial_t u_i - \Delta u_i) = \sum_{j = 1}^M F^{j}_{i}(u,\nabla u) \xi_j, \quad 1 \leq i \leq N, \; 1 \leq j \leq M.
\end{equs}
where the $ \xi_j $ are space-time noises and 
the $ F_{i}^j $ are non-linearities depending 
on the solution and its spacial derivatives.
A complete black box has been set up in the series of papers \cite{reg}, \cite{BHZ}, \cite{CH} and 
\cite{BCCH} covering all the equations treated so far including the generalised KPZ equation, describing the most natural evolution on loop space see \cite{KPZg}. 

Let us briefly summarise the content of this theory. Since \cite{Lyons91}, the rough path approach is a way to study
SDEs driven by non-smooth paths with an enhancement of the underlying path
which allows to recover continuity of the solution map. In the case of SPDEs the
enhancement is represented by a model $ (\Pi,\Gamma) $, to which is associated a space of
local Taylor expansions of the solution with new monomials, coded by an abstract
space $\CT$ of decorated trees. These expansions can be viewed as an extension of the controlled rough paths introduced in 
\cite{Gubinelli2004} which are quite efficient for solving singular SDEs. The main idea is to have a local control
of the behaviour of the solution at some base point. For that, one needs a recentering procedure and a way to act on the coefficients when we change the base point. This action is performed by elements of the structure group $ \CG_+ $ introduced in \cite{reg}.  

Then the resolution procedure of \ref{standard equation}  works as follows. One first mollifies the noises $ \xi_j^{(\varepsilon)} $ and constructs canonically its mollified model $ (\Pi^{(\varepsilon)},\Gamma^{(\varepsilon)}) $. In most situations this mollified model fails to converge because of the potentially ill-defined
products appearing in the right hand side of the equation \ref{standard equation}. Therefore one needs to modify the model to obtain convergence. This is where renormalisation enters the picture. The renormalisation group for the space of models has been originaly described in \cite{reg} but its construction is rather implicit and some parts have to be achieved by hand. This formulation has been used in the different works \cite{reg,wong,woKP,HS15,Hos16,SX16}. 

In  \cite{BHZ}, the authors have constructed an explicit subgroup $ \CG_- $ with Hopf algebra techniques. This group gives an explicit formula for the renormalised model and paves the way for the general convergence result obtained in \cite{CH} for a certain class of models called BPHZ models.

The main aim of this paper is to provide the reader with a direct and an easy construction of the renormalisation model without using all the Hopf algebra machinery. For that purpose, we give a recursive description of the renormalised map $M$ acting on a class of decorated trees: $ M = M^{\circ} R $. The map $ R $ performs a local renormalisation whereas the map $ M^{\circ} $ propagates $ R $ inside the decorated tree. In order to obtain a nice expression of the renormalised model in \cite{BHZ}, the authors use a co-interaction property, described in the context of B-series in \cite{CHV,MR2657947,MR2803804},  between the Hopf algebra for the structure group and the one  for the renormalisation group. This co-interaction gives powerful results but one has to work with extended decorations see \cite{BHZ} for the definitions. In that context, the renormalised model is given in \cite[Thm 6.15]{BHZ} by:
\begin{equs} \label{nice action}
\Pi_x^M = \Pi_x M, \quad  \Gamma_{xy}^{M} = \left( \id \otimes \gamma_{xy}^{M} \right) \Deltap, \quad \gamma_{xy}^{M} = \gamma_{xy} M.
\end{equs}
where $ \gamma_{xy} \in \CG_+ $ and $ \Deltap $ is the coproduct associated to the Hopf algebra describing $ \CG_+ $.
 If we want to define a renormalised model with a map of the form $  M = M^{\circ} R $ a new algebraic property is needed. The main assumption is that $ R $ commutes with the structure group which is not true in general for $ M $.
Gathering other properties on $ R $, this allows us to define a renormalised model with an explicit recursive expression on trees. This expression is given by:
\begin{equs} \label{nice action 2}
\Pi_x^M = \Pi_x^{M^{\circ}} R, \quad \Pi_x^{M^{\circ}} \tau \bar \tau = \Pi_x^{M^{\circ}} \tau \; \Pi_x^{M^{\circ}}\bar \tau, \quad \tau, \bar \tau \in \CT,
\end{equs}
and the rest of the definition of the maps
$ \Pi^{M}_x $ and $ \Pi^{M^{\circ}}_x $ is the same as the one given in \cite{reg} for admissible models. The identities \ref{nice action 2} mean that $ R $ is needed as an intermediary step before recovering the multiplicativity of the model.
 By not using the extended decoration, we are not able to give a nice expression of the action of $ M $ on the structure group like in \ref{nice action}. We circumvent this difficulty by providing a relation between $ \Gamma_{xy} $ and $ \Pi_{x} $ which works for smooth models. Thus it is enough to know
$ \Pi^M_x $ in order to define $ \Gamma^{M}_{xy} $.
 We also prove that this construction is more general than the one given by the Hopf algebra which means that all the examples treated so far are under the scope of this new formulation. 
For proving this fact, we use a factorisation
of the coproducts describing the renormalisation. The main idea is to separate the renormalisation happening at the root from the ones happening inside the tree following the steps of\ref{nice action 2}.  This representation allows us to derive a recursive definition of the coproduct extending the one giving in \cite{reg} and covering the two coproducts used for $ \CG_+ $ and $ \CG_- $.
This approach of decomposing complex coproducts into elementary steps echoes the use of pre-Lie structures through grafting operators in \cite{BCFP,BCCH}. In both cases, the caracterisation of the adjoint $ M^{*} $ of $ M $ as a morphism for the grafting operator appears to be a nice way for describing the renormalised
equation.

Finally, let us give a short review of the content of this paper.
In section~\ref{section 2}, we present the main notations
needed for the rest of the paper and we give a recursive construction of the structure group. We start with the recursive formula given in \cite{reg} as a definition and we carry all the construction of the group by using it.
In section~\ref{section 3}, we do the same for the renormalisation group by introducing the new recursive definition described above. We then present the construction of the renormalised model.
In section~\ref{section 4}, we show that the group given in \cite{BHZ} is a particular case of section~\ref{section 3}
 and we derive a recursive formula for the coproducts.
 In section~\ref{section 5}, we illustrate the construction through some classical singular SPDEs and we rank these equations according to their
 complexity by looking at some properties the renormalised model does or does not satisfy. 
  In the appendix, we show that  some 
of the coassiociativity proofs given in \cite{BHZ} can be recovered 
by using the recursive formula for the coproducts.  

\subsection*{Acknowledgements}
{\small
The author is very grateful to Christian Brouder, Martin Hairer, Dominique Manchon and Lorenzo Zambotti for interesting discussions on the topic.
The author also thanks Martin Hairer for financial support from Leverhulme Trust leadership award.
}

\section{Structure Group}
\label{section 2}

In this section, after presenting the correspondence between trees and symbols we provide an alternative construction of the structure group  using recursive formulae and we prove that this construction coincides with the one described in \cite[Sec.~8]{reg}.
\subsection{Decorated trees and symbolic notation}
In this subsection, we recall mainly the notations on decorated trees introduced in \cite{BHZ}.
Let fix a finite set $ \mfL $ of types and $ d \geq 0$ be the space dimension. We consider $ \Trees  $ the space of decorated trees such that every  
$ T^{\Labn}_{\Labe} \in \Trees $ is formed of:
\begin{itemize}
\item an underlying rooted tree $ T $ with node set $ N_T $, edge set $ E_T $ and root $ \rho_{T} $. To each edge $ e \in E_T $, we associate a type $ \Labhom(e) \in \mfL  $ through a map $ \Labhom : E_T \rightarrow \mfL$. 
\item a node decoration $ \Labn : N_T \rightarrow \N^{d+1} $ and an edge decoration $ \Labe : E_T \rightarrow \N^{d+1} $.
\end{itemize} 
Let fix $  B_{\circ} \subset \Trees $ a family of decorated trees and we denote by $ \CT $ its linear span.
Let $ \s \in \N^{d+1} $ a space-time scaling and  a degree assignment $ | \cdot|_{\s} : \mfL \rightarrow \R \setminus \lbrace 0 \rbrace $. We then associate 
to each decorated tree a degree $  | \cdot |_{\s} $. For $ T^{\Labn}_{\Labe} $ we have:
\begin{equs}
| T^{\Labn}_{\Labe} |_{\s} = \sum_{e \in E_T} ( |\Labhom(e)|_{\s}  - | \Labe(e)|_{\s} ) + \sum_{x \in N_T} | \Labn(x) |_{\s},
\end{equs}
where for $ k \in \N^{d+1} $, $   |k |_{\s} = \sum_{i=0}^d \s_i k_i $ and for $ v \in \Z(\mfL), $, $ |v |_{\s} = \sum_{\Labhom \in \mfL} v_{\Labhom} | \Labhom |_{\s} $  . We now introduce the symbolic notation following \cite[Sec~4.3]{BHZ}:
 \begin{enumerate}
 \item An edge of type $ \mfl $ such that
 $ | \mfl |_{\s} < 0$ is a noise and if it has a zero edge decoration is denoted by $ \Xi_{\mfl} $. We assume that the elements of $ B_{\circ} $ contain noise edges with a decoration equal to zero.
 \item An edge of type $ \Labhom $ such that
 $ | \Labhom |_{\s} > 0$ with decoration
 $ k \in \N^{d+1} $ is an abstract integrator and is denoted by $ \CI^{\Labhom}_{k} $. The symbol $ \CI^{\Labhom}_{k} $ is also viewed as  the operation that grafts a tree onto a new root via a new edge with edge decoration $k$
and type $ \Labhom $.
 \item A factor $ X^k $ encodes the decorated tree $ \bullet^{k} $ with $ k \in \N^{d+1} $ which is the tree composed of a single node and a node decoration equal to $ k $. We write $ X_i $ for $ i \in \lbrace 0,...,d \rbrace $
 as a shorthand notation for $ X^{e_i} $ where the $ e_i $ form the canonical basis of $ \N^{d+1} $. The element $ X^0 $ is denoted by $ \one $.
 \end{enumerate}
 The degree $  | \cdot |_{\s} $ creates a splitting among $ \mfL = \mfL_+ \sqcup \mfL_{-}$ where $ \mfL_+ $ is the set of abstract integrators and $ \mfL_- $ is the set of noises.
 The decorated trees $ X^k $ and $ \Xi_{\mfl} $ can be viewed as linear operators on $ \CT $ through the tree product. This product is defined for two decorated trees $ T_{\Labe}^{\Labn} $, $ \tilde{T}_{\tilde{\Labe}}^{\tilde{\Labn}} $ by $ \bar{T}_{\bar \Labe}^{\bar \Labn} = T_{\Labe}^{\Labn}  \tilde{T}_{\tilde{\Labe}}^{\tilde{\Labn}}$
 where $ \bar{T} = T \tilde{T} $ is the tree obtained by identifying $ \rho_{T} $ and $ \rho_{\tilde{T}} $, $ \bar \Labn $ is equal to $ \Labn $ on $ N_{T} \setminus \lbrace \rho_{T} \rbrace $ and to $ \tilde{\Labn} $ on  $ N_{\tilde{T}} \setminus \lbrace \varrho_{\tilde{T}} \rbrace $ with $ \bar  \Labn(\varrho_{\bar T}) =  \Labn(\varrho_{ T}) +
 \tilde \Labn(\varrho_{ \tilde{T}})   $, the edge decoration $ \bar \Labe $ coincides with $ \Labe $ on $ E_T $ and with $ \tilde{\Labe} $ on $ E_{\tilde{T}} $.

   We suppose that the family $ B_{\circ} $ is strongly conforming to a normal complete rule $ \mathcal{R} $ see \cite[Sec 5.1]{BHZ} which is subcritical as defined in \cite[Def. 5.14]{BHZ}. As a consequence the set $ B_{\alpha} = \lbrace \tau\in B_{\circ}  : \; |\tau |_{\s} = \alpha \rbrace $ is finite for every $ \alpha \in \R $ see  \cite[Prop. 5.15]{BHZ}.  We denote by $ \CT_{\alpha} $ the linear span of $ B_{\alpha} $. 
   
   For the sequel, we introduce another family
   of decorated trees $  B_+ $ which conforms to the rule $ \mathcal{R} $. This means that $ B_{\circ} \subset B_+ $ and we have no constraints on the product at the root. Therefore, $ B_+ $ is stable under the tree product. Then, we consider a disjoint copy $ \bar B_+ $ of $ B_+ $ such that $ B_{\circ} \nsubseteq \bar B_+ $ and we denote by $ \hat \CT_{+} $ its linear span. Elements of $ \bar B_+ $ are denoted by $ (T,2)^{\Labn}_{\Labe} $ where $ T^{\Labn}_{\Labe} \in  B_+ $. Another way to distinguish the two spaces is to use colours as in \cite{BHZ}. The $ 2 $ in the notation means that the root of the tree has the colour $ 2 $ and the other nodes are coloured by $ 0 $.    If the root is not coloured by $ 2 $, we denote the decorated tree as $ (T,0)^{\Labn}_{\Labe}= T^{\Labn}_{\Labe} $. The product on $ \hat \CT_+ $ is the tree product in the sense that the product between $ (T,2)^{\Labn}_{\Labe} $ and $ (\tilde{T},2)^{\tilde{\Labn}}_{\tilde{\Labe}} $ is given by $ (\bar{T},2)^{\bar{\Labn}}_{\bar{\Labe}} $.
   We use a different symbol for the edge incident to a root in $ \hat \CT_+ $  having the type $ \Labhom $ and the decoration $ k $: $\hat  \CJ^{\Labhom}_k $ which can be viewed as an operator from $  \CT $ to $ \hat \CT_{+} $ . 
   The space on which we will define a group in the next subsection is $ \CT_{+}  = \hat \CT_{+} / \CJ_{-} $ where $  \CJ_{-}  $ is the ideal of $ \hat \CT_+ $ generated by 
$ \lbrace \hat \CJ_k^{\Labhom}(\tau), \; \tau \in B_{\circ} \; : \; | \hat \CJ_k^{\Labhom}(\tau) |_{\s}\leq 0 \rbrace $. 
We denote by $ \Pi_+ : \hat \CT_+ \rightarrow \CT_+ $ the canonical projection and $  \CJ^{\Labhom}_k $  the operator from $ \CT $ to $ \CT_+ $ coming from $ \hat \CJ^{\Labhom}_k $.
    
 We want a one-to-one correspondence between decorated trees and certain algebraic expressions.
For a decorated tree $ T_{\Labe}^{\Labn} $, we give a mapping to the symbolic notation in the sense that every tree can be obtained as products, compositions of the symbols $ \CI_k^{\Labhom} $, $ X^k $ and $ \Xi_{\mfl} $ defined above.  We first decompose $  T_{\Labe}^{\Labn} $ into a product of planted trees which are trees of the form $  \CI_{k}^{\Labhom}\left( \tau \right)  $ or $ \Xi_{\mfl} $. Planted trees of $ B_{\circ} $ are denoted by $ \hat B_{\circ} $. The tree $T_{\Labe}^{\Labn}   $ has a unique factorisation of the form:
\begin{equ} \label{e:decomptree}
T_{\Labe}^{\Labn}  = \bullet^{n} \tau_1 \tau_2\cdots  \tau_m\;,
\end{equ}
where each $ \tau_i $ is planted. Then, we define a multiplicative map for the tree product $ \Symb $ given inductively by:
\begin{equs}
\Symb(  \bullet^{n} )  = X^n, \quad \Symb( \Xi_{\mfl} ) =  \Xi_{\mfl}, \quad \Symb(\CI_{k}^{\Labhom}\left( \tau \right)   ) =  \CI_{k}^{\Labhom}\left( \Symb(\tau) \right).  
\end{equs}
For an element $ (T,2)^{\Labn}_{\Labe} \in \CT_+$ , we have the same decomposition \eqref{e:decomptree} but now each of the $ \tau_i  $ must be of positive degree which imply that they are of the form $ \CI_k^{\Labhom}(\tau) $. Then we define another multiplicative map $\Symb^{+}$ by:
\begin{equs}
\Symb^{+}( \bullet^{n}) = X^n, \quad \Symb^{+}(\CJ_{k}^{\Labhom}\left( \tau \right)   ) =  \CJ_{k}^{\Labhom}\left( \Symb(\tau) \right)
\end{equs}
where we have identified $  (\bullet,2)^{n} $ with
$ \bullet^{n} $.
 With this map, a decorated tree $ (T,2)^{\Labn}_{\Labe} \in \CT_+ $
is of the form 
\begin{equs}
 X^{n}  \CJ_{k_1}^{\Labhom_1}(T_1) \cdots \CJ_{k_m}^{\Labhom_m}(T_m), \quad T_i \in B_{\circ}, \quad | \CI_{k_i}^{\Labhom_i}(T_i) |_{\s} > 0.
\end{equs}

Until the section \ref{section 4}, we will use only the symbolic notation in order to construct the regularity structures and the canonical model associated to it.

\subsection{Recursive formulation}

In \cite[Prop. 5.39]{BHZ}, $ \CT $ generated by a  subcritical and normal complete rule $ \mathcal{R} $ gives a regularity structure $ \mathscr{T}_{\mathcal{R}} $. We first recall its definition from  \cite[Def. 2.1]{reg}.

\begin{definition} A triple $ \mathscr{T} = 
(A,\CH,G) $ is called a regularity structure
with model space $ \CH $ and structure group $ G $ if
\begin{itemize}
\item $ A \subset \R $ is bounded from below without accumulation points.
\item The vector space $ \CH = \bigoplus_{\alpha \in A} \CH_{\alpha} $ is graded by $ A $ such that each $ \CH_{\alpha} $ is a Banach space.
\item The group $ G $ is a group of continuous
operators on $ \CH $ such that, for every $ \alpha \in A $, every $ \Gamma \in G $ and every
$ \tau \in \CH_{\alpha} $, one has
\begin{equs}
\Gamma \tau - \tau \in \bigoplus_{\beta < \alpha} \CH_{\beta}.
\end{equs}
\end{itemize}
\end{definition}

For $ \mathscr{T}_{\mathcal{R}} $, we get
\begin{equs}
A = \lbrace | \tau |_{\s} : \; \tau \in B_{\circ} \rbrace, \quad \CH = \CT, \quad \CH_{\alpha} = 
\CT_{\alpha}.
\end{equs}
Concerning the structure group, the aim of the rest of this section is to provide a recursive construction using the symbolic notation.
Let denote by $ \CG_+ $: 
\[
\CG_{+}:=\{g\in \CT_+^*: g(\tau_1\tau_2)=g(\tau_1)g(\tau_2), \ \forall \, \tau_1,\tau_2\in  \CT_+\}. 
\]
For any $g\in \CG_+$ we define a linear operator 
$\Gamma_g: \CT \mapsto \CT$ by
\[
\left\{
  \begin{aligned}
&  \Gamma_g  \one=\one, \qquad  \Gamma_g   \Xi_{\mfl} = \Xi_{\mfl}, \qquad \Gamma_g X_i = X_i+g(X_i), \qquad \Gamma_g (\tau \bar{\tau})= (\Gamma_g \tau )(\Gamma_g \bar{\tau}),  \\
 &     \Gamma_g \CI^{\Labhom}_k(\tau)  =  \CI_k^{\Labhom}(\Gamma_g \tau) + \sum_{\ell \in \N^{d+1}} \frac{X^{\ell}}{\ell!}
  g(\mathcal{J}^{\Labhom}_{k+\ell}(\tau))\;,
 \end{aligned}
 \right.
\]
and we extend it multiplicatively to $ \CT $.
\begin{remark}
 The map $ \Gamma_{g} $ is well defined for every $ g \in \CG_+ $ as a map from $ \CT $ into itself because of the fact that the rule $ \mathcal{R} $ is normal see \cite[Def. 5.7]{BHZ}  which implies that for every $ \prod_{i=1}^n \tau_i \in \CT $ one has 
 $ \prod_{i \in J} \tau_i \in \CT $ where $ J $ is a subset of $ \lbrace 1,...,n \rbrace $. Such operation arises in the definition of $ \Gamma_g $ on some product $ \prod_i \CI_{k_i}^{\Labhom_i}(\tau_i) $ where one can replace any $  \CI_{k_i}^{\Labhom_i}(\tau_i) $ by a polynomial $ \sum_{\ell \in \N^{d+1}} \frac{X^{\ell}}{\ell!}
  g(\mathcal{J}^{\Labhom_i}_{k_i+\ell}(\tau_i)) $. Then we use an inductive argument to conclude that $ \CI_{k_i}^{\Labhom_i}(\Gamma_g \tau_i) \in \CT$ when $ \tau_i \in \CT $. 
 \end{remark}
We define the product $ \circ: \CG_{+} \times \CG_{+} \mapsto \CG_{+} $ recursively by: 
 \[
\left\{ \begin{aligned}
& (g_1 \circ g_2)(X_i) = g_1(X_i) + g_2(X_i), \qquad (g_1 \circ g_2)(\tau_1\tau_2)= (g_1 \circ g_2)(\tau_1) (g_1 \circ g_2)(\tau_2),\\
 & (g_1 \circ g_2)(\mathcal{J}^{\Labhom}_k(\tau)) = g_1(\mathcal{J}^{\Labhom}_{k}( \Gamma_{g_2}\tau)) + \sum_{\ell \in \N^{d+1}} \frac{(g_1(X))^\ell}{\ell!} g_2(\mathcal{J}^{\Labhom}_{k+\ell}( \tau )).
  \end{aligned} 
  \right.
 \]
 where we have made the following abuse of notation
 $ (g_1(X))^{\ell} $ instead of $ \prod_{i =0}^{d} 
 g_1(X_i)^{\ell_i} $. we will also write $ \left( X + g(X) \right)^{\ell} $ instead of $  \prod_{i =0}^{d} \left( X_i + g(X_i) \right)^{\ell_i}  $.
\begin{proposition} 
\begin{enumerate}
\item For every $ g \in \CG_{+} $, $ \alpha \in A $, $ \tau \in \CT_{\alpha} $ and multiindex $ k $, we have $ \Gamma_g \tau - \tau \in \bigoplus_{\beta < | \tau |_{\s}} \CT_{\beta}   $ and $ \Gamma_g \CI^{\Labhom}_{k}( \tau ) - \mathcal{I}^{\Labhom}_{k}( \Gamma_g \tau ) $ is a polynomial.

\item The set $(\Gamma_g,g\in \CG_{+})$ forms a group under the composition of linear operators from $\CT$ to $ \CT$. Moreover, this definition coincides with that of \cite[(8.17)]{reg}.
\item For all $ g , \bar{g}  \in \CG_{+}$, one has $ \Gamma_{g} \Gamma_{\bar{g}} = \Gamma_{g \circ \bar{g}} $. $ (\CG_{+},\circ) $ is a group and each element $ g  $ has a unique inverse $ g^{-1} $ given by the recursive formula
 \begin{equation}\label{g-1}
 \left\{  \begin{aligned}
   & g^{-1}(X_i)  = - g(X_i), \qquad g^{-1}(\tau_1\tau_2)=g^{-1}(\tau_1)g^{-1}(\tau_2), \\
   & g^{-1}(\CJ^{\Labhom}_{k}(\tau)) = -  \sum_{\ell \in \N^{d+1}} \frac{(-g(X))^\ell}{\ell!} g(\CJ^{\Labhom}_{k+\ell}(\Gamma_{g^{-1}} \tau )).
   \end{aligned} \right.
  \end{equation}
  The product $\circ$ coincides with that defined in \cite[Def. 8.18]{reg}.
\end{enumerate}
\end{proposition}
\begin{proof}

We prove the first property by induction on the construction of $ \CT $. Let $ g \in \CG_{+} $. The proof is obvious for $ \tau \in \lbrace  \one,X_i,\Xi_{\mfl} \rbrace $. Let $ \tau = \tau_{1} \tau_2 $ then we have
\[
 \Gamma_g \tau_1 \tau_2  = 
 \Gamma_g \tau_1 \left( \Gamma_g \tau_2 - \tau_2 \right) + \left( \Gamma_g \tau_1 - \tau_1 \right) \tau_2 + \tau_1 \tau_2.  
\] 
We apply the induction hypothesis on $ \tau_1 $ and $ \tau_2 $. Let $ \tau = \CI^{\Labhom}_k(\tau') $ then 
the recursive definition of $ \Gamma_g $ gives: 
\[
 \Gamma_g \CI^{\Labhom}_{k}(\tau') = \CI^{\Labhom}_k(\Gamma_g \tau' - \tau') + \CI^{\Labhom}_k(\tau') + \sum_{\ell \in \N^{d+1}} \frac{X^{\ell}}{\ell!}
  g(\mathcal{J}^{\Labhom}_{k+\ell}(\tau')).
\]
We apply the induction hypothesis on $ \tau' $.

Let $ g , \bar{g}  \in \CG_+$, $h=g \circ \bar{g} \in \CG_+$. Simple computations show that
\[
\Gamma_h  \one=\one, \qquad \Gamma_h   \Xi_{\mfl} = \Gamma_{g} \Gamma_{\bar{g}}\Xi_{\mfl}, \qquad \Gamma_h X_i = \Gamma_{g} \Gamma_{\bar{g}}X_i, \qquad \Gamma_h (\tau\bar{\tau})=\Gamma_{g} \Gamma_{\bar{g}}(\tau\bar{\tau}).
\]
We need to check that $\Gamma_{g} \Gamma_{\bar{g}}\CI^{\Labhom}_k(\tau) =\Gamma_{h}\CI^{\Labhom}_k(\tau) $:
\[
  \begin{aligned}
& \Gamma_{g} \Gamma_{\bar{g}}\CI^{\Labhom}_k(\tau) = \Gamma_{g} \left(\CI_k^{\Labhom}(\Gamma_{\bar{g}} \tau) + \sum_{\ell \in \N^{d+1}} \frac{X^{\ell}}{\ell!}
  {\bar{g}}(\mathcal{J}^{\Labhom}_{k+\ell}(\tau))
\right)
  \\ & = \CI^{\Labhom}_{k}(\Gamma_{g} \Gamma_{\bar{g}}\tau) + \sum_{\ell \in \N^{d+1}} \frac{X^{\ell}}{\ell!}
  g(\CJ^{\Labhom}_{k+\ell}(\Gamma_{\bar{g}}\tau))+ \sum_{\ell \in \N^{d+1}} \frac{(X+g(X))^{\ell}}{\ell!}
  {\bar{g}}(\mathcal{J}^{\Labhom}_{k+\ell}(\tau))
 \end{aligned}
\]
while
\[
  \begin{aligned}
& \Gamma_{h}\CI_k^{\Labhom}(\tau)=\CI^{\Labhom}_{k}(\Gamma_{h}\tau)+
\sum_{\ell \in \N^{d+1}} \frac{X^{\ell}}{\ell!}h(\mathcal{J}^{\Labhom}_{k+\ell}(\tau))
\\ & = \CI^{\Labhom}_{k}(\Gamma_{g} \Gamma_{\bar{g}}\tau) + \sum_{\ell \in \N^{d+1}} \frac{X^{\ell}}{\ell!}
\left( g(\mathcal{J}^{\Labhom}_{k+\ell}( \Gamma_{\bar g}\tau)) + \sum_{j \in \N^{d+1}} \frac{(g(X))^j}{j!} \bar g(\mathcal{J}^{\Labhom}_{k+\ell+j}( \tau )) \right)
\\ & = \CI^{\Labhom}_{k}(\Gamma_{g} \Gamma_{\bar{g}}\tau) +  \sum_{\ell \in \N^{d+1}} \frac{X^{\ell}}{\ell!}
 g(\mathcal{J}^{\Labhom}_{k+\ell}( \Gamma_{\bar g}\tau)) + \sum_{\ell \in \N^{d+1} } \frac{(X+g(X))^\ell}{\ell!} \bar g(\mathcal{J}^{\Labhom}_{k+\ell}( \tau ))
.
 \end{aligned}
\]
By comparing the two formulae, we obtain that $\Gamma_{g} \Gamma_{\bar{g}}=\Gamma_{g \circ \bar{g}}$.

Let us show that $\circ$ is associative on $\CG_+$, namely that $g_1\circ(g_2\circ g_3)=(g_1\circ g_2)\circ g_3$; this is obvious if tested on $X$ and on $\tau\bar{\tau}$; it remains to check this formula on $\CI^{\Labhom}_k(\tau)$:
\[
  \begin{aligned}
& g_1\circ(g_2\circ g_3)(\CJ^{\Labhom}_k(\tau))=g_1(\CJ^{\Labhom}_{k}(\Gamma_{g_2 \circ g_3}\tau)) + \sum_{\ell \in \N^{d+1} } \frac{(g_1(X))^\ell}{\ell!} g_2 \circ g_3(\mathcal{J}^{\Labhom}_{k+\ell}(\tau )) 
\\& =g_1(\CJ^{\Labhom}_{k}(\Gamma_{g_2 \circ g_3}\tau)) + \sum_{\ell \in \N^{d+1} } \frac{(g_1(X))^\ell}{\ell!} g_2 (\mathcal{J}^{\Labhom}_{k+\ell}(\Gamma_{g_3}\tau )) \\ & + \sum_{\ell \in \N^{d+1}} \frac{(g_1(X) + g_2(X))^\ell}{\ell!} g_3 (\mathcal{J}^{\Labhom}_{k+\ell}(\tau ))  ,
 \end{aligned}
\]
while
\begin{equs}
  (g_1\circ g_2)\circ g_3 (\CJ_k^{\Labhom}(\tau)) & = g_1 \circ g_2(\CJ_{k}^{\Labhom}(\Gamma_{g_3}\tau)) + \sum_{\ell \in \N^{d+1} } \frac{(g_1 \circ g_2(X))^\ell}{\ell!} g_3(\mathcal{J}^{\Labhom}_{k+\ell}( \tau )) 
 \\ & = g_1(\CJ^{\Labhom}_{k}(\Gamma_{g_2} \Gamma_{g_3}\tau)) + 
 \sum_{\ell \in \N^{d+1} } \frac{(g_1(X))^\ell}{\ell!} g_2(\mathcal{J}^{\Labhom}_{k+\ell}(\Gamma_{g_3} \tau )) \\ & + \sum_{\ell \in \N^{d+1} } \frac{(g_1(X) + g_2(X))^\ell}{\ell!} g_3(\mathcal{J}^{\Labhom}_{k+\ell}( \tau ))
\end{equs}
and again by comparing the two formulae we obtain the claim.

Let us show now that \eqref{g-1} defines the correct inverse in $(\CG_{+},\circ)$. First of all, the neutral element in $\CG_+$ is clearly ${\bf 1}^*(\tau):=\un{(\tau=\one)}$. As usual, the only non-trivial property is that $g\circ g^{-1}(\CI^{\Labhom}_{k}(\tau)) = g^{-1}\circ g(\CI^{\Labhom}_{k}(\tau)) ={\bf 1}^*(\CI^{\Labhom}_{k}(\tau))=0$. We have
\begin{equs}
& g\circ g^{-1}(\CJ^{\Labhom}_{k + \ell}(\tau)) = g(\CJ_{k+ \ell}(\Gamma_{g^{-1}}\tau)) + \sum_{m \in \N^{d+1}} \frac{(g(X))^m}{m!} g^{-1}(\mathcal{J}^{\Labhom}_{k+\ell + m}( \tau )) , \\
& = g(\CJ^{\Labhom}_{k+ \ell}(\Gamma_{g^{-1}}\tau)) - \sum_{m,\ell \in \N^{d+1}} \frac{(g(X))^m}{m!}  \frac{(-g(X))^\ell}{\ell!} g(\mathcal{J}^{\Labhom}_{k+\ell + m}(\Gamma_{g^{-1}} \tau ))  = 0
 \end{equs}
and
\begin{equs}
& g^{-1}\circ g(\CJ^{\Labhom}_{k}(\tau)) = g^{-1}(\CJ^{\Labhom}_{k}(\Gamma_{g}\tau)) + \sum_{\ell \in \N^{d+1}} \frac{(g^{-1}(X))^\ell}{\ell!} g(\mathcal{J}^{\Labhom}_{k+\ell}( \tau )) 
\\ & = g^{-1}(\CJ^{\Labhom}_{k}(\Gamma_{g}\tau)) + \sum_{\ell \in \N^{d+1} } \frac{(-g(X))^\ell}{\ell!} g(\mathcal{J}^{\Labhom}_{k+\ell}( \tau )) 
\\ & =   - \sum_{\ell \in \N^{d+1} } \frac{(-g(X))^\ell}{\ell!} g(\mathcal{J}^{\Labhom}_{k+\ell}(\Gamma_{g^{-1} \circ g} \tau ))  + \sum_{\ell \in \N^{d+1} } \frac{(-g(X))^\ell}{\ell!} g(\mathcal{J}^{\Labhom}_{k+\ell}( \tau ))  ,
\end{equs}
where we have used a recurrence assumption in the identification $\Gamma_{g^{-1}\circ g} \tau=\tau$. Since $\Gamma_{{\bf 1}^*}$ is the identity in $ \CT $, we obtain that $(\Gamma_g, g\in \CG_+)$ also forms a group. 

We show now that these objects coincide with those defined in \cite[section 8]{reg}. 
In \cite{reg,BHZ}, the action of $\CG_+$ on $ \CT $ is defined through the following co-action $ \Delta: \CT \mapsto  \CT \otimes \CT_+ $,
\[
\left\{
 \begin{aligned}
  &  \Delta \one  = \one \otimes \one,    \text{   } \Delta X_i  = X_i \otimes \one +  \one \otimes X_i,\\
   & \Delta \Xi_{\mfl}  = \Xi_{\mfl} \otimes  \one 
 , \text{   }
  \Delta(\tau \bar{\tau})   = (\Delta \tau) (\Delta \bar{\tau}) \\ &  \Delta \CI^{\Labhom}_k(\tau) =  (\CI^{\Labhom}_k \otimes \id)    \Delta \tau + \sum_{\ell \in \N^{d+1}} \frac{X^\ell}{\ell! } \otimes 
  \mathcal{J}^{\Labhom}_{k+\ell}(\tau).
  \end{aligned}
  \right.
\]
We claim that 
\begin{equation}\label{claimGamma}
 \Gamma_{g} \tau  = (\id \otimes g) \Delta \tau, \qquad \forall\, g\in \CG_+, \, \tau\in \CT.
\end{equation}
First, \eqref{claimGamma} is easily checked on $\one,X_i,\Xi_{\mfl}$ and $\tau\bar\tau\in \CT$. We check the formula on $\CI^{\Labhom}_k(\tau)$:
\[
 \begin{aligned} 
  \Gamma_g \CI^{\Labhom}_k(\tau) & =  \CI^{\Labhom}_k(\Gamma_g \tau) + \sum_{\ell \in \N^{d+1}} \frac{(g(X))^{\ell}}{\ell!}
  g(\mathcal{J}^{\Labhom}_{k+\ell}(\tau)) \\
  & = (\id \otimes g)\left[(\CI^{\Labhom}_k \otimes \id)    \Delta \tau +\sum_{\ell \in \N^{d+1}} \frac{X^\ell}{\ell! } \otimes  \mathcal{J}^{\Labhom}_{k+\ell}(\tau)\right]= (\id \otimes g)\Delta \CI^{\Labhom}_k(\tau).
 \end{aligned}
\]
In \cite{reg,BHZ}, another coproduct $ \Deltap :\CT_+ \mapsto \CT_+\otimes \CT_+$ is defined as follows:
   \[
   \left\{
   \begin{aligned}
 &  \Deltap \one =  \one \otimes \one ,\quad
   \Deltap X_i = X_i  \otimes \one + \one \otimes X_i, \quad
   \Deltap (\tau \bar{\tau}) =
(\Deltap \tau) (\Deltap \bar{\tau}) , \\ 
&   \Deltap \mathcal{J}^{\Labhom}_{k}(\tau) = ( \mathcal{J}^{\Labhom}_{k} \otimes \id )  \Delta \tau +
\sum_{\ell \in \N^{d+1}} \frac{X^\ell}{\ell !} \otimes \CJ^{\Labhom}_{k+ \ell}(\tau) . 
    \end{aligned}
    \right.
   \]
   In order to prove that the product $\circ$ is the same as in \cite{reg}, we need to check that for every $ g_1, g_2 \in \CG_+ $ we have: 
  \[
  \left( g_{1} \circ g_2 \right)(\tau) = \left( g_{1} \otimes g_2 \right) \Deltap \tau, \qquad \forall \, \tau\in \CT_+.
   \]
As usual, this formula is easily checked on $\one ,X_i$ and on products  $\tau\bar\tau\in \CT_+$. We check the formula on $\mathcal{J}^{\Labhom}_k(\tau)$:
   \[
   \begin{aligned}
    & (g_1 \circ g_2)(\mathcal{J}^{\Labhom}_k(\tau)) = g_1(\mathcal{J}^{\Labhom}_{k}(\Gamma_{g_2}\tau)) + \sum_{\ell \in \N^{d+1} } \frac{(g_1(X))^\ell}{\ell!} g_2(\mathcal{J}^{\Labhom}_{k+\ell}( \tau )) \\
    & = (g_{1} \otimes \id) ( \CJ_k^{\Labhom} \otimes g_2) \Delta \tau  
    +  (g_1 \otimes g_2)\sum_{\ell \in \N^{d+1} } \left( \frac{X^\ell}{\ell ! } \otimes \CJ^{\Labhom}_{k+\ell}(\tau) 
   \right)= \left( g_{1} \otimes g_2 \right) \Deltap \tau.
    \end{aligned} 
   \]
   \end{proof}

\section{Renormalised models}

\label{section 3}

We start the section by a general recursive formulation of the renormalisation group without coproduct. Then we use this formulation to construct the renormalised model. During this section, elements of the model space $ \CT $ are described with the symbolic notation.

   \subsection{A recursive formulation}
Before giving the recursive definition of the renormalisation map, we precise some notations. We denote by $ \Vert \tau \Vert $ the number of times the symbols $\Xi_{\mfl}$ appear in $ \tau $. We extend the definitions of $ |\cdot|_{\s} $ and $ \Vert \cdot \Vert $ to any linear combination 
$ \tau=\sum_i \alpha_i \tau_i $ of canonical basis vectors $\tau_i$ with $ \alpha_i \neq 0 $ by
\begin{equ}[e:defNorms]
 \left|\tau \right|_{\s}:= \min_i | \tau_i |_{\s} ,  \qquad \left \Vert \tau \right \Vert:= \max_i \Vert \tau_i \Vert\;,
\end{equ}
which suggests the natural conventions $|0|_{\s} = +\infty$ and $\|0\| = - \infty$.
We also define a partial order  $ <_{\mathcal{T}} $ on $\mathcal{T}$ by setting:
\begin{equation}
\label{order}
\tau_1<_{\mathcal{T}}\tau_2 \quad {\rm if} \quad \|\tau_1\|<\|\tau_2\| \quad {\rm or} \quad
(\|\tau_1\|=\|\tau_2\| \quad {\rm and} \quad |\tau_1|_{\s}<|\tau_2|_{\s}).
\end{equation}

\begin{definition}
A symbol $ \tau $ is an elementary symbol if it has the following form:
$ \Xi_{\mfl} $, $ X_i $ and $ \I^{\Labhom}_k(\sigma) $ where $ \sigma $ is a symbol.
\end{definition}

\begin{proposition} \label{pro:order}
Let $ \tau = \prod_i \tau_i $ such that the $ \tau_i $ are elementary symbols and such that $ \tau $ is not an elementary symbol then $ \tau_i <_{\mathcal{T}} \tau $.
\end{proposition}
\begin{proof}
We consider $ \tau = \prod_i \tau_i $ and let $ \tau_j $ an elementary symbol appearing in the previous decomposition. We define $ \bar{\tau}_{j} =  \prod_{i \neq j } \tau_i $. 
If the product $ \bar{\tau}_{j} $ contains a term of the form $ \CI_{k}^{\Labhom}(\sigma) $ with $ \sigma  $ having at least one noise or a term of the form $ \Xi_{\mfl} $  then $ \Vert \bar{\tau}_j \Vert > 0 $ and  $ \Vert \tau_j  \Vert < \Vert
\bar{\tau}_{j}   \Vert + \Vert \tau_j \Vert = \Vert \tau \Vert $. Otherwise
$ \Vert \tau_j  \Vert = \Vert \tau \Vert  $ but
$  | \bar{\tau}_j |_{\s} > 0 $ which gives 
$  | \tau_j |_{\s}  < | \tau_j|_{\s} + | \bar{\tau}_j |_{\s} = | \tau |_{\s} $. 
Finally, we obtain $ \tau_i <_{\mathcal{T}} \tau $.
\end{proof}

Given a regularity structure $ (A,\CT,G) $, we consider the space  $ \mathcal{L}(\CT) $ of linear maps on $ \CT $. For our recursive formulation, we choose a subset of $ \mathcal{L}(\CT) $:  

\begin{definition}\label{def_adm}
A map $ R \in \mathcal{L}(\CT) $ is admissible if
\begin{enumerate}
\item For every elementary symbol $ \tau $, $ R \tau = \tau $.
\item For every multiindex  $  k $ and any symbol $ \tau $, $ R (X^k \tau)= X^k R \tau $.
\item For each $ \tau \in \mathcal{T} $, $ \Vert R \tau - \tau \Vert < \Vert \tau \Vert $.
 \item For each $ \tau \in \mathcal{T} $, $ | R \tau - \tau |_{\s} > | \tau |_{\s} $.
 \item It commutes with $G$: $R \Gamma = \Gamma R$ for every $\Gamma \in G$.
\end{enumerate}
\end{definition}

We denote by $ \mathcal{L}_{ad}(\CT) $ the set of admissible maps. For $ R \in \mathcal{L}_{ad}(\CT) $, we define a renormalisation map $M=M_R$ by:
\begin{equ}[e:defM]
  \left\{ \begin{aligned}
  & M^{\circ} \one  = \one, \qquad M^{\circ} X_i = X_i,  \qquad M^{\circ} \Xi_{\mfl} = \Xi_{\mfl}, \\ 
  & M^{\circ}  \tau \bar{\tau} = 
 \left( M^{\circ} \tau \right) \left(
 M^{\circ} \bar{\tau} \right), \qquad
 M \tau = M^{\circ} R \tau,    \\
  & M^{\circ} \I_k^{\Labhom}(\tau) =  \I_k^{\Labhom}(M \tau). 
  \end{aligned} \right.
  \end{equ}
The space of  maps $ M $ constructed in this way  is denoted by $ \mathfrak{R}_{ad}[\mathscr{T}] $.
  The main idea behind this definition is that $ R $ computes the interaction between several elements of the product $ \prod_i \tau_i $.
In \cite{reg} and \cite{wong}, elements of the renormalisation group are described by an exponential: $
 M = \exp(\sum_i C_i L_i)$ where $(L_i)_i\subset \mathcal{L}(\CT) $. When the exponential happens to be just equal to $ \id + \sum_i C_i L_i $ then the link with 
 the space $ \mathfrak{R}_{ad}[\mathscr{T}] $ is straightforward. But when several iterations of the $ L_i $ are needed even in the case of the $ L_i $ being commutative the link becomes quite unclear and hard to see. It is also unclear when $ M $ is described with
 a coproduct as in \eqref{Mdelta}. The main difficulty occurs when one has to face nasty divergences.
   The recursive construction is more convenient for several purposes: it gives an explicit and a canonical way of computing the diverging constant we have to subtract. Moreover, the proof of the construction of the renormalised model is simpler than the one given in \cite{reg} and \cite{BHZ}.
 
 \begin{remark}
The Definitions~\ref{e:defNorms}  as well as the convention that follows
are designed in such a way that if the third and the fourth
conditions of Definition~\ref{def_adm} hold for canonical basis vectors $\tau$, then
they automatically hold for every $\tau \in \CT$.
\end{remark}

\begin{remark} 
The first two conditions of Definition~\ref{def_adm} guarantee that $ M $ commutes with the abstract integrator map. 
The third condition is crucial for the definition of $M$: the recursion \eqref{e:defM} stops after a finite number of iterations since it decreases strictly the quantity $ \Vert \cdot \Vert $ and thus the partial order $ <_{\mathcal{T}} $. Moreover, this condition guarantees that $ R = \id + L $ where $ L $ is a nilpotent map and therefore $ R $ is invertible.
The fourth condition allows us to treat the analytical bounds in the definition of the model
and the last condition is needed for the algebraic identities. \end{remark}

\begin{remark}
Note that $M=M_R$ does not always commute with the structure group $G$ even if $ R $ does ;
we will see a counterexample with the group of the generalised KPZ equation in Section\ref{generalised_KPZ_ex}. 
\end{remark}

\begin{proposition} Let $ R \in \mathcal{L}_{ad}(\CT) $, then $ M_R $ is well-defined.
\end{proposition}

\begin{proof}
We proceed by induction using the order $ <_{\mathcal{T}} $. 
If $ \tau \in \lbrace \one,   \Xi_{\mfl}, X_i 
 \rbrace $ then 
$ M  \tau = M^{\circ} R \tau = M^{\circ} \tau = \tau  $. If $ \tau = 
\CI_k^{\Labhom}(\tau') $ then 
\begin{equ} 
M \CI_k^{\Labhom}(\tau') = M^{\circ} R \CI_k^{\Labhom}(\tau') = M^{\circ} \CI_k^{\Labhom}(\tau')
= \CI_k^{\Labhom}(M \tau').
\end{equ}
We conclude by applying the induction hypothesis on $ \tau' $ because we have $ | \tau' |_{\s} < | \tau |_{\s} $.
Let $ \tau = \prod_i \tau_i \in \mathcal{T} $ a product of elementary symbols with at least two symbols in the product, we can write 
\begin{equ}
M \tau = M^{\circ} (R \tau - \tau) 
+ M^{\circ} \tau.
\end{equ}
We apply the induction hypothesis on $ R \tau - \tau <_{\mathcal{T}} \tau $ because $ \Vert R \tau - \tau \Vert < \Vert \tau \Vert $. For $ M^{\circ} \tau  $, we have
\begin{equ}
M^{\circ} \tau = \prod_i M^{\circ} R \tau_i = \prod_i M \tau_i.
\end{equ}
We know from Proposition~\ref{order} that for every $ i $, $ \tau_i <_{\mathcal{T}} \tau $. Therefore, we apply the induction hypothesis on the $ \tau_i $.

\end{proof}

\begin{remark}
In the sequel, we use the order $ <_{\mathcal{T}} $ for all the proofs by induction on the symbols. It is possible to choose other well-order on the symbols for these proofs. One minimal condition in order to have $ M_{R} $ well defined is the following on $ R $: for each $ \tau \in \CT $, there exist $ \Xi_{\mfl_j} $, $ \CI_{k_i}^{\Labhom_i}(\sigma_i) $  and $ k \in \R^{d+1} $ such that
\begin{equs}
 R \tau = X^k \prod_j \Xi_{\mfl_j} \prod_i \CI_{k_i}^{\Labhom_i}(\sigma_i), \quad \sigma_i <_{\CT} \tau. 
 \end{equs}
 Then by definition of $ M_{R} $, we get:
 \begin{equs}
 M^{\circ} R \tau = X^k \prod_j \Xi_{\mfl_j} \prod_i \CI_{k_i}^{\Labhom_i}(M \sigma_i), 
 \end{equs}
 which allows us to use an inductive argument on the $ M \sigma_i $.
\end{remark}
 
 \begin{remark}
There is an alternative definition of the map $ M $.
We denote by $ M_{L} $ the representation of $ M $ given by:
\begin{equation}
\label{recf}
  \left\{ \begin{aligned}
  & M \one  = \one, \qquad M X_i = X_i,  \qquad M \Xi_{\mfl} = \Xi_{\mfl}, \\ 
  & M \prod_i \tau_i = \prod_i M \tau_i  - M L \prod_i \tau_i, \\
  & M \I_k^{\Labhom}(\tau) =  \I_k^{\Labhom}(M \tau), 
  \end{aligned} \right.
  \end{equation}
where the $ \tau_i $ are elementary  and the map $ L $ needs to satisfy the following properties: 
\begin{enumerate}
\item For every elementary symbol $ \tau $, $ L \tau = 0 $ and for every multiindex  $  k $ and symbol $ \bar \tau $, $ L X^k \bar \tau= X^k L \bar \tau $.
\item For each $ \tau \in \mathcal{T} $, $ \Vert L \tau \Vert < \Vert \tau \Vert $ and $ | L \tau |_{\s} > | \tau |_{\s} $.
 \item It commutes with $G$: $L \Gamma = \Gamma L$ for every $\Gamma \in G$.
\end{enumerate}
This properties are very similar to those of $ R $. Noticing that the map $ L $ is nilpotent, one can check that $ R = (\id+L)^{-1} $ and $ M^{\circ} = M (\id+L) $.
\end{remark}

\subsection{Construction of the renormalised Model}

We first define a metric $ d_{\s} $ on $ \R^{d+1} $ associated to the scaling $ \s $ by:
\begin{equs}
d_{\s}(x,y) = \Vert x-y  \Vert_{\s} = \sum_{i=0}^d | x_i - y_i |^{1/\s_i}.
\end{equs}
We recall the definition of a smooth model in \cite[Def. 2.17]{reg}:
\begin{definition} A smooth model for a regularity structure $ \mathscr{T} = 
(A,\CH,G) $ consists of maps:
\[
  \begin{aligned}
  \Pi: \R^{d+1} & \rightarrow \mathcal{L}(\mathcal{H},\mathcal{C}^{\infty}(\R^{d+1})) & & \Gamma: \R^{d+1} \times \R^{d+1}  \rightarrow G \\
   x & \mapsto \Pi_{x} & &  (x,y)  \mapsto \Gamma_{xy}
  \end{aligned}
\]
such that $ \Gamma_{xy} \Gamma_{yz} = \Gamma_{xz} $ and
$ \Pi_{x} \Gamma_{xy} = \Pi_{y} $. Moreover, for every $ \alpha \in A $ and every compact set $ \mathfrak{K} \subset \R^{d+1} $ there exists a constant $ C_{\alpha,\mathfrak{K}} $  such that the bounds 
\[
  \begin{aligned}
  \vert (\Pi_{x} \tau)(y) \vert \leq C_{\alpha,\mathfrak{K}} 
  \|x-y\|_\s^{\alpha} \Vert \tau \Vert_{\alpha}, & & 
  \Vert \Gamma_{xy} \tau \Vert_{\beta} \leq C_{\alpha,\mathfrak{K}} \|x-y\|_\s^{\alpha-\beta} \Vert \tau \Vert_{\alpha}
  \end{aligned}
\]
hold uniformly over all $ (x,y) \in \mathfrak{K} $, all $ \beta \in A $ with $ \beta \leq \alpha $ and all $ \tau \in \mathcal{H}_{\alpha} $.
\end{definition}

In the previous definition, for $ \tau \in \CH $, $ \Vert \tau \Vert_{\alpha} $ denotes the norm of 
the component of $ \tau $ in the Banach space $ \CH_{\alpha} $. We suppose given a collection of kernels $ \lbrace K_{\Labhom} \rbrace_{\Labhom \in \mathcal{L}_+} $, $ K_{\Labhom} : \R^{d+1} \setminus {0} \rightarrow \R $ satisfying the condition \cite[Ass. 5.1]{reg} with $ \beta = | \Labhom|_{\s} $ and a collection of noises $\lbrace \xi_{\mfl} \rbrace_{\Labhom \in \mathcal{L}_-}  $ such that $ \xi_{\mfl} \in \mathcal{C}^{\infty}(\R^{d+1}) $.
We use the notation  $ D^k = \prod_{i=0}^d \frac{\partial^{k_i}}{\partial y_i^{k_i}} $ for $  k \in \R^{d+1} $. Until the end of the section, $ R $ is an admissible map and $ M $ is a renormalisation map built from $ R $.  

As in \cite{reg}, we want a renormalised model $ (\Pi^M_{x},\Gamma_{xy}^M) $ constructed from a map $ \PPi $ satisfying the following property: 
\begin{equ}[e:piM]
 \PPi^M \tau =\PPi M \tau.
\end{equ}
We define the linear map $\PPi^M $ by:
\begin{equ}[e:Pim]
  \left\{ \begin{aligned}
  & (\PPi^{M^{\circ}} \one)(y) = 1 ,  \qquad
  (\PPi^{M^{\circ}} X_i)(y)  = y_i,  \qquad
  (\PPi^{M^{\circ}} \Xi_{\mfl})(y) = \xi_{\mfl}(y), \\
  & (\PPi^{M^{\circ}} \CI^{\Labhom}_{k} \tau)(y) = 
  \int D^{k} K_{\Labhom}(y-z) (\PPi^{M} \tau)(z) dz,
  \\
  &  (\PPi^{M^{\circ}} \tau \bar{\tau} )(y) =  (\PPi^{M^{\circ}} \tau )(y) 
  (\PPi^{M^{\circ}} \bar{\tau} )(y) , \qquad (\PPi^{M} \tau )(y) = 
(\PPi^{M^{\circ}} R  \tau )(y), 
  \end{aligned} \right.
\end{equ}
where the recursive definition is the same as for $ M $.  The definition of $ \PPi^M  $ is really close to the definition of $\PPi $. The main difference is that $\PPi^M $ is no longer multiplicative because we have to renormalise some ill-defined products by subtracting diverging terms which is performed by the action of $ R $.
\begin{remark}
We have chosen the definition \eqref{e:Pim} for $\PPi^{M} $ instead of \eqref{e:piM} because it contains the definition of $\PPi $ when $ R = \id $.
Moreover, the recursive formula for the product is really close to the definition of $ \Pi_x^{M} $ and this fact is useful for the proofs.
\end{remark}

\begin{proposition} We have the following identities: $\PPi^M \tau =\PPi M \tau
 $ and $\PPi^{M^{\circ}} \tau   =\PPi M^{\circ} \tau$.
\end{proposition}
\begin{proof}
We proceed again by induction.  It's obvious for $ \one $, $ X_{i} $ and $ \Xi_{\mfl} $.
For $ \tau =\I^{\Labhom}_{k}(\tau') $,  by the induction hypothesis the claim holds for $ \tau' $ 
because $ \Vert \I^{\Labhom}_k(\tau') \Vert = \Vert \tau' \Vert $ and $ | \I^{\Labhom}_k(\tau') |_{\s} > | \tau' |_{\s}. $ 
We have: 
\begin{equs}
 (\PPi^{M} \CI_{k}^{\Labhom} (\tau'))(y)  & = 
  \int D^{k} K_{\Labhom}(y-z) (\PPi^{M} \tau')(z) dz  =  \int D^{k} K_{\Labhom}(y-z) (\PPi M \tau')(z) dz \\
  & = (\PPi \CI^{\Labhom}_{k}( M \tau'))(y) = (\PPi M \CI^{\Labhom}_{k} (\tau'))(y).
  \end{equs}
For $ \tau = \prod_i \tau_i $ product of elementary symbols, we obtain by applying the induction hypothesis on $ R \tau - \tau $ and the $ \tau_i $: 
\begin{equs}
(\PPi^{M^{\circ}}  \tau )(y) = \prod_i (\PPi^{M^{\circ}}  \tau_i )(y) = \prod_i (\PPi {M^{\circ}}  \tau_i )(y) =
 (\PPi {M^{\circ}}  \tau )(y)
\end{equs}
and 
\begin{equs}
 (\PPi^{M} \tau )(y)&  = 
  (\PPi^{M^{\circ}} (R\tau - \tau) )(y) + (\PPi^{M^{\circ}}  \tau )(y)    \\
& = (\PPi {M^{\circ}} (R\tau - \tau) )(y) +  (\PPi {M^{\circ}}  \tau )(y) = (\PPi M \tau )(y) 
\end{equs}
which conclude the proof.
\end{proof}

 The  renormalised model $ (\Pi^{M},\Gamma^{M}) $ associated to $ M=M_R $ is given by  
 \[
 \label{e:Pi}
\left\{
\begin{aligned}
& (\Pi_x^{M^{\circ}} \one)(y) = 1  , \qquad (\Pi_{x}^{M^{\circ}} \Xi_{\mfl})(y)  = \xi_{\mfl}(y) , \qquad
  (\Pi_{x}^{M^{\circ}} X_i)(y) = y_i-x_i, 
\\ &  (\Pi_{x}^{M^{\circ}} \CI^{\Labhom}_{k} \tau)(y) =  
  \int D^k K_{\Labhom}(y-z) \Pi^{M}_{x}(\tau)(z) dz - \sum_{\ell \in \N^{d+1}} 
  \frac{(y-x)^{\ell}}{\ell!} f_{x}^{M}({\mathcal J}^{\Labhom}_{k+\ell}(\tau)), \\
&   (\Pi_{x}^{M}  \tau )(y) = 
  (\Pi_{x}^{M^{\circ}} R  \tau )(y), \qquad
    (\Pi_{x}^{M^{\circ}}  \tau \bar{\tau} )(y) = 
   (\Pi_{x}^{M^{\circ}}   \tau )(y)  (\Pi_{x}^{M^{\circ}}   \bar{\tau} )(y),
   \end{aligned} \right.
\]

where $ f_{x}^{M} \in \CT_+^* $  is defined by 
\[
\left\{ \begin{aligned}
  & f_{x}^{M}(X_i) = x_i, \qquad  f_{x}^{M}(\tau \bar{\tau}) = f_{x}^{M}(\tau) f_{x}^{M}(\bar{\tau}),
\\  & f_{x}^{M}({\mathcal J}^{\Labhom}_{k}(\tau))  =  \un{(|{\mathcal I}^{\Labhom}_{k}(\tau)|_{\s}>0)}
  \int  D^{k} K_{\Labhom}(x-z) (\Pi_{x}^{M}\tau)(z) dz.
  \end{aligned} \right.
\]

We also define
\[
\left\{   \begin{aligned}
   & \Gamma_{xy}^M X_i = X_i + (x_i-y_i), \qquad \Gamma_{xy}^M \Xi_{\mfl} = \Xi_{\mfl}, \qquad
   \Gamma_{xy}^M (\tau \bar{\tau}) = ( \Gamma_{xy}^M \tau)  ( \Gamma_{xy}^M \bar{\tau}), \\
   & \Gamma_{xy}^M \CI_k^{\Labhom}(\tau) 
   = \CI_k^{\Labhom}(\Gamma_{xy}^M \tau) - \sum_{\ell \in \N^{d+1}} \frac{(X+x-y)^{\ell}}{\ell!}  f_{y}^{M}(\mathcal{J}^{\Labhom}_{k+\ell}(\tau)) +\sum_{\ell \in \N^{d+1}} \frac{X^\ell}{\ell!}
  f_{x}^{M}(\mathcal{J}^{\Labhom}_{k+\ell}(\Gamma_{xy} \tau )).
   \end{aligned} \right.
\]
and
\[
\left\{ \begin{aligned}
  & g^M_{x}(X_i) = -x_i, \qquad  g^M_{x}( \tau \bar \tau) =  g^M_{x}( \tau ) g^M_x(\bar \tau), 
\\  & g^M_{x}( \CJ_k^{\Labhom}(\tau))  =  -\sum_{\ell \in \N^{d+1} } \frac{(-x)^\ell}{\ell!} f^M_x(\CJ^{\Labhom}_{k+\ell}(\tau)).
  \end{aligned} \right.
\]

\begin{proposition}
The $ \Gamma^M $ operator is also given by: 
\[
  \Gamma^M_{xy} = (F^M_x)^{-1} \circ F^M_{y}
\]
where $ F^{M}_{x} = \Gamma_{g_x^M} $. Moreover, another equivalent recursive definition is:
\begin{equation}\label{another_gamma}
\left\{   \begin{aligned}
   & \Gamma_{xy}^M X_i = X_i + (x_i-y_i), \qquad \Gamma_{xy}^M \Xi_{\mfl} = \Xi_{\mfl}, \qquad
   \Gamma_{xy}^M (\tau \bar{\tau}) = ( \Gamma_{xy}^M \tau)  ( \Gamma_{xy}^M \bar{\tau}) \\
   & \Gamma_{xy}^M \CI^{\Labhom}_k(\tau) 
   = \CI^{\Labhom}_k(\Gamma_{xy}^M \tau) - \sum_{|\ell|_{\s}<| \CI^{\Labhom}_k(\tau) |_{\s} } (\Pi_{x}^M  \CI^{\Labhom}_{k+\ell}( \Gamma_{xy}^M \tau))(y) \frac{(X+x-y)^{\ell}}{\ell !} .
   \end{aligned} \right.
\end{equation}
\end{proposition}
\begin{proof} We have
\[
\begin{aligned}
(g_{x}^{M})^{-1}(\mathcal{J}^{\Labhom}_{k}(\tau)) =& 
-  \sum_{\ell \in \N^{d+1} } \frac{(-g_{x}^{M}(X))^\ell}{\ell!} g_{x}^{M}(\mathcal{J}^{\Labhom}_{k+\ell}(\Gamma_{(g_{x}^{M})^{-1}} \tau ))
\\ =& 
-  \sum_{\ell,m \in \N^{d+1}} \frac{(x)^\ell}{\ell!} \frac{(-x)^m}{m!} f_{x}^{M}(\mathcal{J}^{\Labhom}_{k+\ell + m }(\Gamma_{(g_{x}^{M})^{-1}} \tau )) \\
= & -   f_{x}^{M}(\mathcal{J}^{\Labhom}_{k }(\Gamma_{(g_{x}^{M})^{-1}} \tau )).
 \end{aligned}
\]
Since by definition
\[
\Gamma_g \CI^{\Labhom}_k(\tau)  =  \CI^{\Labhom}_k(\Gamma_g \tau) + \sum_{\ell \in \N^{d+1}} \frac{X^{\ell}}{\ell!}
  g(\mathcal{J}^{\Labhom}_{k+\ell}(\tau)),
\]
then
\[
\Gamma_{g_{y}^{M}} \CI_k^{\Labhom}(\tau)  =   \CI_k^{\Labhom}(\Gamma_{g_{y}^{M}} \tau) - \sum_{\ell \in \N^{d+1}} \frac{(X-y)^{\ell}}{\ell!}   f_{y}^{M}(\mathcal{J}^{\Labhom}_{k+\ell}(\tau))
\]
and
\[
\begin{aligned}
\Gamma_{(g_{x}^{M})^{-1}} \CI_k^{\Labhom}(\tau)  & =  \CI_k^{\Labhom}(\Gamma_{(g_{x}^{M})^{-1}} \tau) + \sum_{\ell \in \N^{d+1}} \frac{X^{\ell}}{\ell!}
  (g_{x}^{M})^{-1}(\mathcal{J}^{\Labhom}_{k+\ell}(\tau))
  \\ & = \CI^{\Labhom}_k(\Gamma_{(g_{x}^{M})^{-1}} \tau) - \sum_{\ell \in \N^{d+1}} \frac{X^{\ell}}{\ell!}
  f_{x}^{M}(\mathcal{J}^{\Labhom}_{k+l}(\Gamma_{(g_{x}^{M})^{-1}} \tau )),
  \end{aligned}
\]
so that
\[
\begin{aligned}
\Gamma_{(g_{x}^{M})^{-1}}\Gamma_{g_{y}^{M}}  \CI_k^{\Labhom}(\tau)  = &  \CI_k^{\Labhom}(\Gamma_{(g_{x}^{M})^{-1}}\Gamma_{g_{y}^{M}} \tau) - \sum_{\ell \in \N^{d+1}} \frac{X^{\ell}}{\ell!}
  f_{x}^{M}(\mathcal{J}^{\Labhom}_{k+l}(\Gamma_{(g_{x}^{M})^{-1}}\Gamma_{g_{y}^{M}} \tau ))
\\ & 
+ \sum_{\ell \in \N^{d+1}} \frac{(X+x-y)^{\ell}}{\ell!}   f_{y}^{M}(\mathcal{J}^{\Labhom}_{k+\ell}(\tau)). 
\end{aligned}
\]
Therefore, $\Gamma_{(g_{x}^{M})^{-1}}\Gamma_{g_{y}^{M}}$ satisfies the same recursive property as $\Gamma^M_{xy}$. 

Finally, we need to prove \eqref{another_gamma}. We have
\[
 \Gamma^{M}_{xy} \CI^{\Labhom}_{k}(\tau) = 
 \CI^{\Labhom}_{k}(\Gamma^{M}_{xy} \tau) 
  +  \sum_{\ell \in \N^{d+1}} \frac{(X+x-y)^{\ell}}{\ell!}  A^{M}_{y,x,k,\ell},
\]
where 
\[
A^{M}_{y,x,k,\ell} = f^M_{y}(\mathcal{J}^{\Labhom}_{k+\ell}(\tau)) - \sum_{m \in \N^{d+1}} \frac{(y-x)^m}{m!} 
  f^M_{x}(\mathcal{J}^{\Labhom}_{k+\ell+m}( \Gamma^M_{xy} \tau)).
\]
We write $ \Gamma^{M}_{xy} \tau = \sum_i \tau_i $ with $ |\tau_i |_{\s} \leq |\tau|_{\s} $; 
note that $\Pi^M_{y} \tau=\Pi^M_{x} \Gamma_{xy}^M \tau=\sum_i \Pi^M_{x} \tau_i $, and
$A^{M}_{y,x,k,\ell}$ is zero unless $ | \mathcal{J}^{\Labhom}_{k+\ell}(\tau) |_{\s} > 0 $, and if this condition is satisfied 
then
\[
\begin{aligned}
& A^{M}_{x,y,k,\ell} =\int D^{k+\ell} K_{\Labhom}(y-z) (\Pi^M_{y} \tau)(z)  dz   - \sum_i \sum_{m \in \N^{d+1}} \frac{(y-x)^m}{m!} 
  f^M_{x}(\mathcal{J}^{\Labhom}_{k+\ell+m}(  \tau_i))
\\ &  =  \sum_i \left[ \int D^{k+\ell} K_{\Labhom}(y-z) (\Pi^M_{x} \tau_i)(z)  dz - \sum_{m \in \N^{d+1}} \frac{(y-x)^m}{m!} 
  f^M_{x}(\mathcal{J}^{\Labhom}_{k+\ell+m}(  \tau_i))\right]
 \\ & = \sum_i \Pi_x^{M}(\mathcal{I}^{\Labhom}_{k+\ell}(\tau_i))(y)
 = \Pi_x^M(\CI^{\Labhom}_{k+\ell}(\Gamma^{M}_{xy} \tau))(y).
  \end{aligned}
\]
This allows us to conclude.
\end{proof}

\begin{remark} The interest of the previous formula for $ \Gamma^M $ is to show a strong link with the definition of $ \Pi_x^M $. Moreover it simplifies the proof of the analytical bounds of the model. Indeed, analytical bounds on $ \Pi_x^M $ give the bounds for $ \Gamma^M $.  
\end{remark}

\begin{proposition}
The following identities hold: $ \Pi_x^M =\PPi^M F_x^M $ and 
 $ \Pi_x^{M^{\circ}} =\PPi^{M^{\circ}} F_x^M $. 
\end{proposition}
\begin{proof}
We proceed by induction. The proof is obvious for $ \tau \in \lbrace \one,\Xi_{\mfl},X_i \rbrace $. For $ \tau = \CI^{\Labhom}_k(\tau') $, we apply the induction hypothesis on $ \tau' $, it follows:  
\[
\begin{aligned}
 (\Pi_x^{M^{\circ}} \tau )(y) & = 
 \int D^k K_{\Labhom}(y-z) (\Pi^{M}_{x}\tau')(z) dz - \sum_{\ell \in \N^{d+1}} 
  \frac{(y-x)^{\ell}}{\ell!} f_{x}^{M}({\mathcal J}^{\Labhom}_{k+\ell}(\tau')) \\
  & =  \int D^k K_{\Labhom}(y-z) (\Pi^{M} F_x^M \tau')(z) dz - \sum_{\ell \in \N^{d+1}} 
  \frac{(y-x)^{\ell}}{\ell!} f_{x}^{M}({\mathcal J}^{\Labhom}_{k+\ell}(\tau')) \\
  & =( \Pi^M F_x^M \CI^{\Labhom}_k(\tau'))(y).
  \end{aligned}
\]
It remains to check the identity on a product $ \tau = \prod_i \tau_i $ where each $\tau_i$ is elementary.
We have 
 \begin{equ}
\PPi^M F_x^M \tau =  \PPi^{M^{\circ}} R   F_x^M \tau =\PPi^{M^{\circ}}    F_x^M R \tau
 \end{equ}
since by definition $F^{M}_{x} = \Gamma_{g_x^M} \in G $ and $ R $ commutes with $ G $. Then by applying the induction hypothesis on $ R \tau - \tau $ and  the $ \tau_i $, we have
\begin{equs}
\Pi^{M^{\circ}}    F_x^M  \tau = \prod_i\PPi^{M^{\circ}}    F_x^M  \tau_i = \prod_i\PPi_x^{M^{\circ}}     \tau_i = \PPi_x^{M^{\circ}}     \tau
\end{equs}
and
\begin{equs}
\PPi^{M^{\circ}}    F_x^M R \tau & =\PPi^{M^{\circ}}    F_x^M (R \tau - \tau) +\PPi^{M^{\circ}}    F_x^M  \tau 
\\ & = \PPi_x^{M^{\circ}}     (R \tau - \tau) + \PPi_x^{M^{\circ}}    \tau
= \PPi_x^{M^{\circ}}     R \tau  = \PPi_x^{M}     \tau.
\end{equs}

\end{proof}

\begin{proposition}
If $ R $ is an admissible map then $ (\Pi^M,\Gamma^M) $ is a model.
\end{proposition}
\begin{proof}
The algebraic relations are given by the previous proposition. It just remains to check the analytical bounds. For $ \tau = \Xi_{\mfl} $ or $ \tau = 
\CI_k^{\Labhom}(\tau')$, the proof is the same as in \cite[Prop. 8.27]{reg}. For $ \tau = \prod_i \tau_i $ a product of elementary symbols, we have
\[
\Pi_x^M \tau =  \Pi_x^{M^{\circ}} (R \tau - \tau ) +   \Pi_x^{M^{\circ}} \tau . 
\]
We apply the induction hypothesis on the $\tau_i $ and $ R \tau - \tau$: 
\begin{equs}
  \vert (\Pi^{M^{\circ}}_{x} \tau)(y) \vert = \prod_i \vert (\Pi^{M^{\circ}}_{x} \tau_i)(y) \vert \lesssim  
  \prod_i \Vert x-y \Vert_{\s}^{| \tau_i|_{\s}} = \Vert x-y \Vert_{\s}^{| \tau|_{\s}}  ,
  \\    \vert (\Pi^{M^{\circ}}_{x} (R \tau - \tau))(y) \vert \lesssim  
 \Vert x-y \Vert_{\s}^{\vert R \tau - \tau \vert_{\s}} \lesssim \Vert x-y \Vert_{\s}^{\vert \tau \vert_{\s}} .
 \end{equs}

It just remains the analytical bound for $ \Gamma^M $. We proceed by induction. For $ \tau = \Xi_{\mfl} $ or $ \tau = X_i $, the bound is obvious. Let $ \tau = \prod_i \tau_i $ where the $ \tau_i $ are elementary symbols. For $ \beta < | \tau |_{\s} $, we have
\[
\begin{aligned}
 \Vert \Gamma_{xy}^M \tau \Vert_{\beta} & = \sum_{\sum_{i} \alpha_{i} = \beta \atop 
 \alpha_{i} < \vert \tau_i \vert_{\s} }  \prod_i \Vert \Gamma_{xy}^M
 \tau_i \Vert_{\alpha_{i}} \\
 & \lesssim  \sum_{\sum_{i} \alpha_{i} = \beta \atop 
 \alpha_{i} < \vert \tau_i \vert_{\s} }  \prod_i   \| x-y\|_\s^{\vert \tau_i \vert_{\s} - \alpha_{i}  } \Vert \tau \Vert \lesssim \| x-y\|_\s^{\alpha - \beta}.
 \end{aligned}
\]
For $ \tau' = \mathcal{I}^{\Labhom}_{k}(\tau) $, the recursive definition \eqref{another_gamma} gives:
\[
 \Gamma_{xy}^M \CI^{\Labhom}_k( \tau) =  \CI^{\Labhom}_k(\Gamma_{xy}^M \tau) - \sum_{|\ell|_{\s} < |\CI^{\Labhom}_k( \tau)|_{\s}} (\Pi^M_{x}  \CI^{\Labhom}_{k+\ell}( \Gamma_{xy}^M \tau))(y) \frac{(X+x-y)^{\ell}}{\ell !} .
\]
 Let $\alpha<|\CI^{\Labhom}_k( \tau)|_{\s}$. If $ \alpha \in \R \setminus \N $, let us write $\Gamma_{xy}^M \tau=\tau+\sum_i \tau^i_{xy}$ with 
 \[
 |\tau^i_{xy}|_{\s}=\alpha_i<|\tau|_{\s}, \qquad \|\tau^i_{xy}\|_{\alpha_i}\lesssim \| x-y\|_\s^{|\tau|_{\s}-\alpha_i} ;
 \]
  then if 
\begin{equs}
\Vert \Gamma_{xy}^M \tau' \Vert_{\alpha} = 
\Vert \CI^{\Labhom}_k(\Gamma_{xy}^M \tau) \Vert_{\alpha} & \lesssim\sum_i\un{(\alpha_i+ | \Labhom |_{\s}-|k|_{\s}=\alpha)}\| x-y\|_\s^{|\tau|_{\s}-\alpha_i} \\ & \lesssim \|x-y\|_\s^{|\tau|_{\s}+| \Labhom |_{\s} - |k|_{\s} -\alpha}.
\end{equs}
Now, if $ \alpha \in \N $ and $ \alpha < |\CI^{\Labhom}_k( \tau)|_{\s}$ then 
 \[
 \begin{aligned}
  & \Vert \Gamma_{xy}^M \CI^{\Labhom}_k( \tau) \Vert_{\alpha}   = \left\vert 
 \sum_{ \alpha \leq | \ell |_{\s} < | \Labhom |_{\s} + |\tau|_{\s} -|k|_{\s} } \frac{(X+x-y)^{\ell}}{\ell !} (\Pi^M_{x}  \CI^{\Labhom}_{k+\ell}( \Gamma_{xy}^M \tau))(y) \right\vert  \\
 & \lesssim \sum_{ \alpha \leq | \ell |_{\s} < | \Labhom |_{\s} + |\tau|_{\s} -|k|_{\s} } \frac{\|x-y\|_\s^{ |\ell|_{\s}-\alpha}}{\ell!} \sum_{\gamma \leq  | \tau |_{\s}} \|x-y\|_\s^{| \Labhom |_{\s} + \gamma -|k|_{\s}-|\ell|_{\s}} \Vert \Gamma_{xy}^M  \tau\Vert_{| \Labhom |_{\s}+\gamma-|k|_{\s}} \\
 & \lesssim  \sum_{ \alpha \leq | \ell |_{\s} < | \Labhom |_{\s} + |\tau|_{\s} -|k|_{\s} } \frac{\|x-y\|_\s^{| \ell |_{\s}-\alpha}}{\ell !} \sum_{\gamma \leq  | \tau |_{\s}} \|x-y\|_\s^{| \Labhom |_{\s} + \gamma -|k|_{\s}-|\ell|_{\s}} \|x-y\|_\s^{|\tau|_{\s} - | \Labhom |_{\s} - \gamma + |k|_{\s}}  \\ 
 & \lesssim \|x-y\|_\s^{|\tau|_{\s}-\alpha}.
 \end{aligned}
 \]
\end{proof}

\begin{proposition} \label{prop_PiM}
We suppose that for every $ \tau = \mathcal{I}_{k}(\tau') \in\mathcal{T}$ such that $ |\tau |_{\s} < 0 $, we have $ (\Pi_x^M \tau)(x) = (\Pi_x M \tau)(x) $.  
Then
the following identities hold: $ (\Pi_x^M \tau )(x) = (\Pi_{x} M \tau )(x) $ and $ (\Pi_x^{M^{\circ}} \tau )(x) = (\Pi_{x} M^{\circ} \tau )(x) $  for every $ \tau \in \mathcal{T} $. 
\end{proposition}
\begin{proof}
We proceed by induction. For $ \tau \in \lbrace \one,\Xi_{\mfl},X_{i} \rbrace $, we have 
\[
 (\Pi_x^M \tau )(x) =  (\Pi_x \tau )(x) = (\Pi_x M \tau )(x).
\]
For $ \tau = \CI^{\Labhom}_{k}( \tau')  $ , if $ | \tau |_{\s} > 0 $ then the recursive definition of $ \Pi_x^M $ gives 
\[
\begin{aligned}
 (\Pi_{x}^{M} \CI^{\Labhom}_{k} \tau')(x)  & = 0\\
 ( \Pi_x M  \CI^{\Labhom}_{k} \tau')(x)&  = ( \Pi_x   \CI^{\Labhom}_{k} M \tau')(x) = 0.
  \end{aligned}
\]
For the second identity, we have used the fact that $ | M \tau' |_{\s} \geq |\tau' |_{\s} $.
Otherwise, if $ |\tau|_{\s} < 0 $ then the hypothesis allows us to conclude.
For an elementary product $ \tau =  \prod_i \tau_i $, it follows by using the induction hypothesis
\begin{equs}
(\Pi_x^{M^{\circ}} \tau)(x) & =  \prod_i (\Pi_x^{M^{\circ}}  \tau_i)(x) = \prod_i (\Pi_x M^{\circ} \tau_i)(x) =  (\Pi_x {M^{\circ}} \tau)(x)
 \\
(\Pi_x^{M} \tau)(x) & = (\Pi_x^{M^{\circ}} (R \tau - \tau) )(x) + (\Pi_x^{M^{\circ}} \tau)(x)
\\ & =  (\Pi_x {M^{\circ}} (R \tau - \tau) )(x) + (\Pi_x M^{\circ} \tau)(x) = (\Pi_x M \tau)(x).
\end{equs}

\end{proof}

\begin{remark}
Proposition \ref{prop_PiM} is crucial for deriving the renormalised equation in many examples. Indeed, the reconstruction map $ \mathcal{R}^{M} $ associated to the model $ (\Pi^M,\Gamma^M) $ is given for every $ \tau \in \mathcal{T} $ by: 
\[
 (\mathcal{R}^{M} \tau )(x) = (\Pi^M_{x} \tau)(x) = (\Pi_x M \tau)(x)
\]
because for every $ \tau \in \mathcal{T} $, $ \Pi_x \tau $ is a function. The result of Proposition~\ref{prop_PiM} has just been checked on  examples 
\cite{reg}, \cite{wong} and \cite{woKP} but not in a general setting. In general for $y\ne x$,  $ (\Pi_x^M \tau)( y)   $ is not necessarily equal to $(\Pi_x M \tau)(y)$ as mentioned in  \cite{reg}. But if you carry more information
on the decorated tree with an extended decoration, then this identity turns to be true see \cite[Thm. 6.15]{BHZ}. 
\end{remark}

We finish this section by establishing a link between the renormalisation maps introduced in \cite{reg}  and $  \mathfrak{R}_{ad}[\mathscr{T}] $.
From \cite[Lem.~8.43, Thm~8.44]{reg} and \cite[Thm~B.1]{woKP}, 
${\mathfrak R} \subset \mathcal{L}(\CT)$ is the set of maps $M$ such that
\begin{itemize}
\item One has $\CI_k^\Labhom M\tau =M \CI_k^\Labhom\tau$ and $MX^k \tau=X^k M \tau$
for all $\Labhom\in\Lab_+$, $k\in\N^{d+1}$, and $\tau \in \CT$.
\item Consider the (unique) linear operators $\DeltaM:\CT\to\CT\otimes\CT_+$ and $\hat M:\CT_+\to\CT_+$ 
such that $\hat M$ is an algebra morphism,  $\hat M X^k=X^k$ for all $k$, and such that,
for every $\tau, \sigma \in \CT$ and $k \in \N^{d+1}$ with $|\CJ^\Labhom_k(\sigma)|_{\s}>0$,
\begin{equs}
\hat M \tilde\CJ^\Labhom_k(\sigma) =\CM_+(\tilde\CJ^\Labhom_k&\otimes\id)\DeltaM\sigma\;,  \label{e:hatM} \\
(\id \otimes\CM_+)(\Delta\otimes\id)\DeltaM \tau &=(M\otimes\hat M)\Delta\tau\;, \label{e:hatM2}
\end{equs}
where $\tilde\CJ^\Labhom_k \colon \CT \to \CT_+$ is defined for every $ \tau \in \CT $ by $ \tilde\CJ^\Labhom_k (\tau) = \sum_{\ell \in \N^{d+1}} \frac{(-X)^{\ell}}{\ell !} \CJ^\Labhom_{k+ \ell} (\tau)  $ and $ \CM_+ : \CT_+ \otimes \CT_+ \rightarrow \CT_+ $ is the product on $ \CT_+ $, $ \CM_+(\tau_1 \otimes \tau_2) = \tau_1 \tau_2 $.
Then, for all $\tau\in\CT$, one can write
$\DeltaM \tau= \sum \tau^{(1)}\otimes\tau^{(2)}$
with $| \tau^{(1)}|_{\s} \ge | \tau|_{\s}$. Having this latest property, $ \DeltaM $ is called an upper triangular map.
\end{itemize}

Let $ M \in \mathfrak{R}_{ad}[\mathscr{T}] $, we 
build two linear maps $\DeltaM$ and $\DeltaB$ by setting
\begin{equ}
\DeltaB \one = \one\otimes \one\;,\qquad \DeltaB X_i = X_i \otimes \one\;,\qquad \DeltaB \Xi_{\mfl} = \Xi_{\mfl} \otimes \one\;,
\end{equ}
and then recursively
\begin{equ}[e:DeltaM]
\DeltaB \tau\bar \tau = \bigl(\DeltaB \tau\bigr)\bigl(\DeltaB \bar\tau\bigr)\;,\qquad
\DeltaM \tau = \DeltaB R\tau\;,
\end{equ}
as well as
\begin{equ}[e:IDeltaB]
\DeltaB \CI_k^\Labhom(\tau) = (\CI_k^\Labhom \otimes \id)\DeltaM \tau - \sum_{| \ell|_{\s} \ge |\CI^\Labhom_k\tau|_{\s}} {X^{\ell} \over \ell!} \otimes  \CM_+ \bigl(\tilde{\CJ}^\Labhom_{k+\ell} \otimes \id\bigr)\DeltaM \tau\;.
\end{equ}
We claim that if $ \DeltaB $ and $\DeltaM$ are defined in this way, then provided that one defines
$\hat M$ by \eqref{e:hatM}, the identity \eqref{e:hatM2} holds.

\begin{proposition} \label{propDeltaM}
If $M \in \mathfrak{R}_{ad}[\mathscr{T}]$, $\DeltaM$ is defined as above and $ \hat M  $ is defined by \eqref{e:hatM}, then the identity 
\eqref{e:hatM2} holds and $M$ belongs to $\RR$.
\end{proposition}

Before giving the proof of Proposition~\ref{propDeltaM}, we need to rewrite \eqref{e:hatM2}. Indeed, the identity
\eqref{e:hatM2} is equivalent to 
\begin{equ}[e:defDeltaM]
\DeltaM = (\id \otimes \CM_+) \bigl((\id \otimes \CA_+)\Delta M \otimes \hat M\bigr)\Delta \;,
\end{equ}
where $ \CA_+ : \CT_{+} \rightarrow \CT_{+} $ is the
antipode associated to $ \Deltap $ defined by: 
 \begin{equs}
 \mathcal{M}_{+} \left( \id \otimes \CA_+  \right) \Deltap = \one^* \one  =  \mathcal{M}_{+} \left(  \CA_+ \otimes \id \right) \Deltap. 
 \end{equs}
 The identity \eqref{e:defDeltaM} is the consequence of the following lemma:

\begin{lemma}
Let $ D : \CT \otimes \CT_{+} \rightarrow \CT \otimes \CT_{+} $ given by
\begin{equs}
D  = (\id \otimes \CM_+)(\Delta \otimes \id)
\end{equs}
then $ D $ is invertible and $ D^{-1} $ is given by 
\begin{equs}
D^{-1} = (\id \otimes \CM_+) (\id \otimes \CA_+ \otimes \id)(\Delta \otimes \id).
\end{equs}
\end{lemma}
\begin{proof}
We have by using the fact that  $ (\Delta \otimes \id) \Delta = (\id \otimes \Deltap) \Delta $   and $ \CM_+ (\id \otimes \CM_+) = 
 \CM_+ (\CM_+ \otimes \id) $
\begin{equs}
D^{-1} D & = (\id\otimes \CM_+) (\id\otimes \CA_+ \otimes \id) (\id \otimes \id \otimes \CM_+)( (\Delta \otimes \id)\Delta \otimes \id)  \\
& =  (\id \otimes \CM_+ \left( \CA_+ \otimes \CM_+ \right) )  ( (\id \otimes \Deltap)\Delta \otimes \id) \\
& = (\id \otimes \CM_+)  (\id \otimes \CM_+ \left( \CA_+ \otimes \id \right) \otimes \id )  ( (\id \otimes \Deltap)\Delta \otimes \id) \\
& = (\id \otimes \CM_+)   ( (\id \otimes \CM_+ \left( \CA_+ \otimes \id \right)\Deltap)\Delta \otimes \id). 
\end{equs}
Now it follows with the identities $ \CM_+ \left( \CA_+ \otimes \id \right)\Deltap = \one^* \one $ and $ (\id \otimes \one^{*})\Delta \tau = (\tau \otimes \one) $
\begin{equs}
D^{-1} D 
& = (\id \otimes \CM_+)   ( (\id \otimes \one^{*})\Delta \otimes \id) = \id \otimes \id.
\end{equs}
Using the same properties, we prove that $ D D^{-1} = \id \otimes \id $.
\end{proof}

\begin{remark}
The previous lemma gives an explicit expression of the inverse of $ D $. It is a refinement of \cite[Proposition 8.38]{reg} 
which proves the fact that $ D $ is invertible.
\end{remark}

\begin{remark}\label{rem:strong}
The equivalence between \eqref{e:defDeltaM} and \eqref{e:hatM2} is in the strong sense
that \eqref{e:defDeltaM} holds for any given symbol $\tau$ if and only if \eqref{e:hatM2} holds
for the same symbol $\tau$.
\end{remark}
Before giving the proof of Proposition~\ref{propDeltaM}, we provide some identities concerning the antipode $ \CA_+ $.
Regarding the antipode $\CA_+$, one has the recursive definition
\begin{equs}
& \CA_+ \one = \one, \quad \CA_+ X_i = -X_i, \quad
\CA_+(\tau_1 \tau_2) = \CA_+(\tau_1) \CA_+(\tau_2), \\ & \CM_+ \left( \id \otimes \CA_+   \right) \Deltap \mathcal{J}^{\Labhom}_{k} \tau = 0.
\end{equs}
which gives
\begin{equ}[e:defA]
\sum_{\ell \in \N^{d+1}} \frac{X^{\ell}}{\ell!}\CA_+ \CJ^{\Labhom}_{k+ \ell }(\tau) = - \CM_+ \Bigl(\CJ^{\Labhom}_{k } \otimes \CA_+\Bigr)\Delta\tau\;.
\end{equ}
 As a consequence of this, one has the identity
\begin{equ}[e:magic]
 \CM_+ \bigl(\CA_+ \CJ^{\Labhom}_{k} \otimes \id \bigr)\Delta \tau = -\tilde{\CJ}^{\Labhom}_{k}(\tau)
\;.
\end{equ}
To see this, simply apply on both sides in \eqref{e:defA} the antipode $ \CA_+ $.

\begin{proof}[Proposition~\ref{propDeltaM}]
Since \eqref{e:hatM} holds by definition and it is straightforward to verify that 
$\DeltaM$ is upper triangular (just proceed by induction using \eqref{e:IDeltaB}
and \eqref{e:DeltaM}), 
we only need to verify that \eqref{e:hatM2}, or equivalently \eqref{e:defDeltaM}, holds.
For this, we first note that since $\DeltaM = \DeltaB R$, $M = M^\circ R$,
$R$ commutes with $\Delta$, and since $R$ is invertible by assumption, 
\eqref{e:defDeltaM} is equivalent to the identity
\begin{equ}[e:wantedDR]
\DeltaB = (\id \otimes \CM_+) \bigl(( \id \otimes \CA_+)\Delta M^\circ \otimes \hat M\bigr)\Delta\;,
\end{equ}
and it is this identity that we proceed to prove now.
Both sides in \eqref{e:wantedDR} are morphisms so that, 
by induction, it is sufficient to show that 
if \eqref{e:wantedDR} holds for some element $\tau$, then it also holds for $\CI^{\Labhom}_k(\tau)$. 
(The fact that it holds for $\one$, $X_i$ and $\Xi_{\mfl}$ is easy to verify.)

 Starting from \eqref{e:IDeltaB}, we first use \eqref{e:hatM} and the fact that 
$\DeltaM$ and $\DeltaB$ agree on elements of the form $\CI_k^{\Labhom}(\tau)$
to rewrite $\DeltaB \CI_k^{\Labhom}(\tau)$ as
\begin{equ}[e:start]
\DeltaB \CI^{\Labhom}_k(\tau) = (\CI_k \otimes \id)\DeltaM \tau + \sum_{\ell \in \N^{d+1}} {X^{\ell} \over \ell!} \otimes \bigl(\hat M \tilde{\CJ}^{\Labhom}_{k+ \ell}(\tau) - \CM_+ (\tilde{\CJ}^{\Labhom}_{k+ \ell}\otimes \id)\DeltaM\tau\bigr)\;,
\end{equ}
where the sum runs over all multiindices $\ell$ (but only finitely many terms in the sum
are non-zero). By  \eqref{e:magic}, we have 
\begin{equs}
\CM_+ (\tilde{\CJ}^{\Labhom}_{k+ \ell}\otimes \id)\DeltaM\tau 
&=-   \CM_+ \Bigl(\CM_+(\CA_+ \CJ^{\Labhom}_{k+\ell} \otimes \id)\Delta\otimes \id\Bigr)\DeltaM\tau\\
&=-   \CM_+ \bigl(\CA_+ \CJ^{\Labhom}_{k+\ell} \otimes \id\bigr) \bigl(\id \otimes \CM_+\bigr)\bigl(\Delta\otimes \id\bigr)\DeltaM\tau\;,
\end{equs}
Recall that by Remark~\ref{rem:strong}, the induction hypothesis implies that
\eqref{e:hatM2} holds, so that we finally conclude that 
\begin{equ}
\CM_+ (\tilde{\CJ}^{\Labhom}_{k+ \ell}\otimes \id)\DeltaM\tau  = -  \CM_+ \bigl(\CA_+ \CJ_{k+ \ell}^{\Labhom}M \otimes \hat M\bigr)\Delta\tau\;.
\end{equ}

Using again the induction hypothesis, but this time in its form \eqref{e:defDeltaM},
we thus obtain from \eqref{e:start}  the identity
\begin{equs}
\DeltaB \CI^{\Labhom}_k(\tau) &= (\id \otimes \CM_+) \bigl((\CI^{\Labhom}_k\otimes \CA_+)\Delta M \otimes \hat M\bigr)\Delta\tau + \sum_{\ell \in \N^{d+1}} {X^{\ell} \over \ell!} \otimes  \hat M \tilde{\CJ}^{\Labhom}_{k + \ell}(\tau) \\
&\quad + \sum_{\ell \in \N^{d+1}} {X^{\ell} \over \ell!} \otimes   \CM_+ \bigl(\CA_+ \CJ^{\Labhom}_{k+\ell}M \otimes \hat M\bigr)\Delta\tau \;.
\end{equs}
At this stage, we see that we can use the definition of $\Delta$ to combine the first and the last term, yielding
\begin{equs}
 \DeltaB &  \CI^{\Labhom}_k(\tau) = (\id \otimes \CM_+) \bigl((\id\otimes \CA_+)\Delta \CI^{\Labhom}_k M \otimes \hat M\bigr)\Delta\tau  +  \sum_{\ell \in \N^{d+1}} {X^{\ell} \over \ell!} \otimes  \hat M \tilde{\CJ}^{\Labhom}_{k+ \ell}(\tau)\\
&= (\id \otimes \CM_+) \bigl((\id\otimes \CA_+)\Delta  M^\circ  \otimes \hat M\bigr)(\CI^{\Labhom}_k\otimes \id)\Delta\tau + \sum_{\ell \in \N^{d+1}} {X^{\ell} \over \ell!} \otimes  \hat M \tilde{\CJ}^{\Labhom}_{k+\ell}(\tau)\;.
\end{equs}
We now rewrite the last term
\begin{equs}
\sum_{\ell \in \N^{d+1}} {X^{\ell} \over \ell!} &\otimes  \hat M \tilde{\CJ}^{\Labhom}_{k+ \ell}(\tau)
= \sum_{m,\ell \in \N^{d+1}} (\id \otimes\CM_+)\Bigl({X^{\ell} \over \ell!} \otimes {(-X)^{m} \over m!} \otimes  \hat M \CJ^{\Labhom}_{k+ \ell + m}(\tau)\Bigr)\\
&\quad= \sum_{\ell \in \N^{d+1}} (\id \otimes\CM_+)\Bigl((\id \otimes \CA_+)\Delta M^\circ{X^\ell \over \ell!} \otimes \hat M \CJ^{\Labhom}_{k+\ell}(\tau)\Bigr)\\
&\quad= \sum_{\ell \in \N^{d+1}} (\id\otimes\CM_+)\bigl(( \id\otimes \CA_+)\Delta M^\circ \otimes  \hat M \bigr)\Bigl({X^\ell \over \ell!} \otimes \CJ^{\Labhom}_{k+\ell}(\tau)\Bigr)\;.
\end{equs}
Inserting this into the above expression and using the definition of $\Delta$ 
finally yields \eqref{e:wantedDR} as required, thus concluding the proof. 
\end{proof}

\section{Link with the renormalisation group}
\label{section 4}

In this section, we establish a link between the renormalisation group $ \mathfrak{R}[\mathscr{T}] $
defined in \cite{BHZ} and 
the maps $ M $ constructed from admissible maps $ R $.

\begin{theorem} \label{inclusion theorem}
One has  $ \mathfrak{R}[\mathscr{T}] \subset \mathfrak{R}_{ad}[\mathscr{T}]$.
\end{theorem}
The outline of this section is the following. We first start by recalling the definition of $ \mathfrak{R}[\mathscr{T}] $ and then we show that these maps are of the form $ M^{\circ} R $. Then we prove the commutative property with the structure group and we  give a proof of Theorem~\ref{inclusion theorem}. Finally, we derive a recursive formula.

\subsection{The renormalisation group}

 Before giving the definition of $ \mathfrak{R}[\mathscr{T}] $, we need to introduce some notations. Let $ \hat\CT_- $ the free commutative algebra generated by $ B_{\circ}$. We denote by $ \cdot $ the forest product associated to this algebra. Elements of $ \hat\CT_- $ are of the form  $ (F,\Labn,\Labe) $ where $ F $ is now a forest.
 Then  for the forest product we have:
 \begin{equs}
 (F,\Labn,\Labe) \cdot (G,\bar \Labn,\bar \Labe)  = (F \cdot G,\bar \Labn + \Labn,\bar \Labe + \Labe)
 \end{equs} 
where the sums $ \bar \Labn + \Labn $ and  $ \bar \Labe + \Labe $ mean that decorations defined on one of the forests are extended to the disjoint union by setting them to vanish on the other forest.
 Then we set $ \CT_{-} = \hat \CT_{-} / \CJ_{+} $ where $  \CJ_{+}  $ is the ideal of $ \hat \CT_- $ generated by 
$ \lbrace \tau \in B_{\circ} \; : \; | \tau |_{\s}\geq 0 \rbrace $. Then the map $ \Deltam :  \CT \rightarrow \CT_- \otimes \CT $ defined in \cite{BHZ} is given for $ T^{\Labn}_{\Labe} \in \CT $ by: 
\begin{equs}\label{co-action minus}
 \Deltam  T^{\Labn}_{\Labe} & = 
 \sum_{A \in \Adm(T) } \sum_{\Labe_A,\Labn_A}  \frac1{\Labe_A!}
\binom{\Labn}{\Labn_A}
 (A,\Labn_A+\pi\Labe_A, \Labe \restr E_A) 
 \\& \qquad  \otimes( \CR_A T, [\Labn - \Labn_A]_A, \Labe + \Labe_A)\;, 
\end{equs}
where
\begin{itemize}
\item For $ C \subset D $ and $ f : D \rightarrow \N^{d+1} $, let $ f \restr C$ the restriction of $ f $ to $ C $.
\item The first sum runs over $ \Adm(T) $,  all subgraphs $ A  $ of
$ T $, $ A $ may be empty. The second sum runs over all  $ \Labn_A : N_A \rightarrow \N^{d+1}  $ and $ \Labe_{A} : \partial(A,T) \rightarrow \N^{d+1} $ where 
$ \partial(A,F) $ denotes the edges in $ E_T \setminus E_A $ that are adjacent to $ N_A $.
\item  We write $ \CR_A T $ for the tree obtained by contracting the connected components of $ A $. Then we have an action on the decorations in the sense that for $ f : N_T \rightarrow \N^{d+1}$ such that $ A \subset T $ one has: $ [f]_A(x) = \sum_{x \sim_{A} y} f(y) $ where $ x $ is an equivalence class of $ \sim_A $ and
$ x \sim_A y $ means that $ x $ and $ y $ are connected in $ A $. For $ g : E_T  \rightarrow \N^{d+1}  $, we define for every $ x \in N_T $, $ (\pi g)(x) = \sum_{e=(x,y) \in E_T} g(x)$.
\end{itemize}

Then one can turn this map into a coproduct 
$ \Deltam : \CT_-  \rightarrow \CT_{-} \otimes \CT_{-}$ and obtain a Hopf algebra for $  \CT_{-}  $ endowed with this coproduct and the forest product see \cite[Prop. 5.35]{BHZ}.  The main difference here is that we do not consider extended decorations but the results for the Hopf algebra are the same as in \cite{BHZ}. This allows us to consider the group of character of this Hopf algebra denoted by $ \mathcal{G}_- $ which is the set of multiplicative 
elements of $ \CT_-^* $ the dual of $ \CT_- $ endowed with the convolution product $ \circ $ described below: 
\begin{equs}
\ell \circ \bar \ell = \left( \ell \otimes \bar \ell \right) \Deltam, \quad \ell^{-1} = \ell(\CA_- \cdot), \quad \ell, \bar \ell \in \mathcal{G}_-,
\end{equs}
where $ \CA_- $ is the antipode for $ \Deltam $. In comparison to \cite{BHZ}, we add the assumptions on $ \CG_- $ that every $ \ell \in \CG_- $ is zero on $ \hat B_{\circ} \cap B_{\circ}^{-} $ and on every $T^{\Labn}_{\Labe} \in B_{\circ}^{-}  $ such that $ \Labn(\rho_T) \neq 0 $ where $  \hat B_{\circ} $ are the planted trees of $ B_{\circ} $ and $ B^{-}_{\circ} $ are the elements of $ B_{\circ} $ with negative degree. Actually it is easy to check that these assumptions define a subgroup of $ \CG_- $.
Then the renormalisation group $ \mathfrak{R}[\mathscr{T}] $ is given by:
\begin{equs} \label{Mdelta}
\mathfrak{R}[\mathscr{T}] = \lbrace M_{\ell} = \left( \ell \otimes \id \right) \Deltam, \; \ell \in \CG_- \rbrace.
\end{equs}
In order to rewrite these maps $ M_{\ell} $ in terms of some $ M^{\circ} $ and $ R $, we need to derive a factorisation of the map $ \Deltam $. We denote the forest product by $ \cdot $ and the empty forest by $ \one_1 $.

\begin{definition}
For all trees $T$ we denote by $\Adm^\circ(T)$ the family of all (possibly empty) 
$A\in \Adm(T)$ such that if $A=\{S_1,\ldots,S_k\}$ then $\rho_{S_i}\ne\rho_T$ for all $i=1,\ldots,k$. For all forests $F$ with $F=T_1\cdot T_2\cdots T_n$ we set 
\begin{equs}
\Adm^\circ(F):=\{ A\in\Adm(F): \, A= A_1 \sqcup \cdots  \sqcup A_m, \,  A_i\in
\Adm^\circ(T_i), \, i=1,\ldots,m\},
\end{equs}
and $\Adm^\circ(\one):=\emptyset$. We also define $ \Adm^r(T) $ containing the empty forest and the non-empty forests  $A $ such that $ A = \lbrace S \rbrace $ and $ \rho_S = \rho_T $.
\end{definition}

We define the maps $ \Deltam_{\circ} $ (resp. $ \Deltam_r $) as the same as $ \Deltam  $ by replacing $ \Adm(T) $ in \eqref{co-action minus}
by $ \Adm^{\circ}(T) $(resp. $ \Adm^r(T) $).

\begin{proposition} \label{mult_Mcirc}
The map $\Deltam_{\circ}$ is multiplicative on $\CT$, i.e. $\Deltam_{\circ}(\tau_1\tau_2)=(\Deltam_{\circ}\tau_1)(\Deltam_{\circ}\tau_2)$ for all $\tau_1,\tau_2\in\CT$, where we consider on $\CT_-\otimes\CT$ the product 
\[
(\phi_1,\tau_1)\otimes(\phi_2,\tau_2)\mapsto(\phi_1\cdot\phi_2,\tau_1\tau_2).
\]
\end{proposition}
\begin{proof}
We introduce $\hat\Pi:\hat B_{\circ}\mapsto  B_{\circ}$ the map which associates to a planted decorated tree $T^\Labn_\Labe$ a decorated tree obtained by erasing the only edge $e=(\rho,y)$ incident to the root $\rho$ in $T$ and setting the root to be $y$.

Since all decorated trees are products of elementary trees, it is enough to prove that for $\tau_1,\ldots,\tau_m\in\hat B_{\circ}$ such that
$ \tau_1 ... \tau_m \in B_{\circ} $ we have $\Delta_{\circ}(\tau_1\cdots\tau_m)=(\Delta_{\circ}\tau_1)\cdots(\Delta_{\circ}\tau_m)$. It is easy to see that for all $T^\Labn_\Labe\in\hat B_{\circ}$
\begin{equs} \label{recursive_identity_Mcirc}
\Adm^{\circ}(T^\Labn_\Labe)=\Adm(\hat\Pi T^\Labn_\Labe), \qquad \Deltam_{\circ} T^\Labn_\Labe=(\id \otimes X^{\Labn(\rho)}\I^{\Labhom(e)}_{\Labe(e)})\Deltam\hat\Pi T^\Labn_\Labe
\end{equs}
where as above $\rho$ is the root of $T^\Labn_\Labe$ and $e=(\rho,y)$ is the only edge incident to $\rho$. 

Now if $T^\Labn_\Labe=\tau_1\cdots\tau_m\in B_{\circ}$ with $\tau_1,\ldots,\tau_m\in\hat B_{\circ}$, then we have a canonical bijection between $\Adm^{\circ}(T^\Labn_\Labe)$ and $\Adm(\hat\Pi \tau_1)\times\cdots\times\Adm(\hat\Pi\tau_m)$ and the numerical coefficients factorise nicely, so that
\begin{equs}
 \Deltam_\circ( \tau_1\cdots\tau_m) 
& = \prod_{i=1}^m\sum_{A_i\in \Adm(\hat\Pi \tau_i)}
\sum_{\Labe_{A_i},\Labn_{A_i}} {1\over \Labe_{A_i}!}\binom{\Labn}{\Labn_{A_i}} \, (A_i, \Labn_{A_i} + \pi \Labe_{A_i}, \Labe \restr E_{A_i} ) \\ &   \otimes X^{\Labn(\rho_i)}\I^{\Labhom(e_i)}_{\Labe(e_i)} ( \CR_{A_i} \hat\Pi \tau_i, [ \Labn - \Labn_{A_i} ]_{A_i}, \Labe+\Labe_{A_i} )
\\ & = \prod_{i=1}^m (\id \otimes X^{\Labn(\rho_i)}\I^{\Labhom(e_i)}_{\Labe(e_i)})\Deltam \hat\Pi \tau_i=
\prod_{i=1}^m \Deltam_{\circ} \tau_i.
\end{equs}
\end{proof}

\begin{proposition} \label{factorisation_admt}
Let $ {\mathcal M}_-:\CT_{-}\otimes \CT_- \mapsto \CT_-$, $\phi\otimes\bar\phi\mapsto\phi\cdot\bar\phi$. Then
\begin{equs}
({\mathcal M}_{-} \otimes \id)(\id \otimes \Deltam_{\circ})\Deltam_r=\Deltam
\end{equs}
\end{proposition}
holds on $ \hat \CT_{-} $.
\begin{proof}
By multiplicativity on $\hat \CT_-$, it is enough to prove the equality on all $T_\Labe^\Labn\in B_{\circ}$. Note that
\begin{equs}
(\id \otimes \Deltam_\circ)\Deltam_r T_\Labe^\Labn= &\sum_{A\in\Adm^r(T)} \sum_{\Labe_A,\Labn_A} {1\over \Labe_A!} \binom{\Labn}{\Labn_A}\sum_{B\in\Adm^\circ(\CR_A T)} \sum_{\Labe_B,\Labn_B} {1\over \Labe_B!} \binom{ [\Labn - \Labn_A]_A}{\Labn_B} \cdot
\\
&\quad \cdot \left( (A, \Labn_A + \pi \Labe_A,  \Labe \restr A ) \cdot
  (B, \Labn_B + \pi \Labe_B,\Labe+\Labe_A \restr B) \right)
 \\ & \otimes ( \CR_B \CR_A T,   [[ \Labn-\Labn_A-\Labn_B ]_A ]_B,\Labe+\Labe_A + \Labe_B)
\end{equs}
and therefore   
\begin{equs}
( &{\mathcal M}_{-}\otimes \id)( \id \otimes\Deltam_\circ)\Deltam_r \\ = 
 &\sum_{A\in\Adm^r(T)} \sum_{\Labe_A,\Labn_A} {1\over \Labe_A!} \binom{\Labn}{\Labn_A}\sum_{B\in\Adm^\circ(\CR_A T)} \sum_{\Labe_B,\Labn_B} {1\over \Labe_B!} \binom{[\Labn - \Labn_A]_A}{\Labn_B} \cdot
\\
&\quad \cdot (A, \Labn_A + \pi \Labe_A,  \Labe \restr A )\otimes 
  (B, \Labn_B + \pi \Labe_B,\Labe+\Labe_A \restr B)
 \\ & \otimes ( \CR_B \CR_A T,  [[ \Labn-\Labn_A-\Labn_B ]_A]_B,\Labe+\Labe_A + \Labe_B)
\end{equs}
At this point, we note that since $B \in \Adm^\circ(\CR_A T)$, 
$\Labe_A$ and $\Labe_B$ have disjoint support so that 
$\Labe_A!\Labe_B! = (\Labe_A + \Labe_B)!$. Similarly, thanks to the fact
that $\Labn_B$ has support away from the root of $\CR_A T$, one
has 
\begin{equ}
\binom{[\Labn - \Labn_A]_A}{\Labn_B} = \binom{\Labn - \Labn_A}{\Labn_B}\;,
\end{equ}
so that 
\begin{equs}
\binom{\Labn}{\Labn_A}\binom{[\Labn - \Labn_A]_A}{\Labn_B}
&= 
{\Labn! (\Labn-\Labn_A)! \over \Labn_A! (\Labn-\Labn_A)! \Labn_B! (\Labn-\Labn_A-\Labn_B)!} \\
&=
{\Labn! \over (\Labn_A+ \Labn_B)! (\Labn-\Labn_A-\Labn_B)!}
= 
\binom{\Labn}{\Labn_A+ \Labn_B}\;.
\end{equs}
We note also that the map $(A,B)\mapsto C = A \cup B$ is a bijection between
$\{(A,B):A\in\Adm^r(T), B\in\Adm^\circ(\CR_A T)\}$ and $\Adm(T)$, since every $C \in \Adm(T)$ is either in
$\Adm^\circ(T)$ (if none of the subtrees touches the root of $T$), 
or of the form $\{S\} \cup B$, where $\rho_S=\rho_T$ and $B \in \Adm^\circ(\CR_A T)$. Moreover setting,
$\Labe_C = \Labe_A + \Labe_B$ and $\Labn_C = \Labn_A + \Labn_B$,
the above sum can also be rewritten as
\begin{equs}
 \sum_{C\in\Adm(T)}\sum_{\Labn_C,\Labe_C}
{1\over \Labe_C!} \binom{\Labn}{\Labn_C}  
\,(C, \Labn_C + \pi \Labe_C, \Labe \restr C ) 
\\   \otimes
 (\CR_C T, [ \Labn - \Labn_C ]_C,\Labe + \Labe_C) = \Deltam T_\Labe^\Labn.
 \end{equs}
This concludes the proof.
\end{proof}

\begin{corollary}\label{recursive_f}
Let $\ell \in \mathcal{G}_{-}$, $R_\ell \eqdef (\ell \otimes \id)\Deltam_r$  and $M^{\circ}_\ell \eqdef (\ell \otimes \id)\Deltam_{\circ}$. Then $ M_{\ell} = M^\circ_{\ell} R_{\ell} $.
\end{corollary}
\begin{proof}
Note that
\begin{equs}
M^\circ_{\ell} R_{\ell} = 
\left( \ell \otimes \id \right) \Deltam_{\circ} \left( \ell \otimes \id \right) \Deltam_r =\left( \ell \otimes \ell \otimes \id \right) \left( \id \otimes \Deltam_{\circ} \right) \Deltam_r.
\end{equs}
Now, since $\ell: \CT_-  \mapsto\R$ is multiplicative, we obtain
\begin{equs}
M^\circ_{\ell} R_{\ell} = (\ell\ \otimes \id)({\mathcal M}_{-}\otimes \id)( \id \otimes \Deltam_{\circ})\Deltam_r =(\ell \otimes \id)\Deltam =M_\ell.
\end{equs}
\end{proof}

 Through the next tree $ T$, we illustrate the Proposition \ref{factorisation_admt} by considering $ A =  \lbrace S_1, S_2, S_3 \rbrace \in \Adm(T) $: 

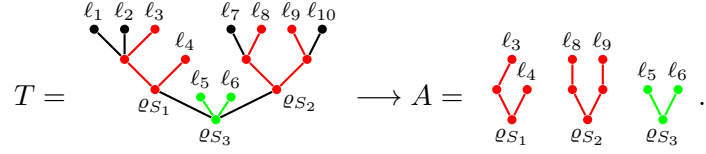
\begin{equs}
 T =  \begin{tikzpicture}[scale=0.2,baseline=0.2cm]
        \node at (0,0)  [dot,green,label={[label distance=-0.2em]below: \scriptsize  $ \rho_{S_3} $}] (root) {};
         \node at (-8,6)  [dot,label={[label distance=-0.2em]above: \scriptsize  $ \ell_1 $} ] (leftll) {};
          \node at (-1,1.5)  [dot,green,label={[label distance=-0.2em]above: \scriptsize  $ \ell_5 $} ] (center) {};
           \node at (1,1.5)  [dot,green,label={[label distance=-0.2em]above: \scriptsize  $ \ell_6 $} ] (center1) {};
          \node at (-6,4)  [dot,red,label=left:] (leftl) {};
          \node at (-4,6)  [dot,red,label={[label distance=-0.2em]above: \scriptsize  $ \ell_3 $}] (leftlr) {};
          \node at (-6,6)  [dot,label={[label distance=-0.2em]above: \scriptsize  $ \ell_2 $}] (leftlc) {};
      \node at (-4,2)  [dot,color=red,label={[label distance=-0.2em]below: \scriptsize  $ \rho_{S_1} $}] (left) {};
         \node at (-2,4)  [dot,red,label={[label distance=-0.2em]above: \scriptsize  $ \ell_4 $}] (leftr) {};
         \node at (2,4)  [dot,red,label=above:   ] (rightl) {};
          \node at (1,6)  [dot,label={[label distance=-0.2em]above: \scriptsize  $ \ell_7 $} ] (rightll) {};
           \node at (3,6)  [dot,red,label={[label distance=-0.2em]above: \scriptsize  $ \ell_8 $}] (rightlr) {};
           \node at (6,4)  [dot,red,label=left:] (rightr) {};
            \node at (5,6)  [dot,red,label={[label distance=-0.2em]above: \scriptsize  $ \ell_9 $}] (rightrl) {};
        \node at (4,2) [dot,red,label={[label distance=-0.4em]  below right: \scriptsize  $ \rho_{S_2} $}] (right) {};
         \node at (7,6)  [dot,label={[label distance=-0.2em]above: \scriptsize  $ \ell_{10} $} ] (rightrr) {};
        
        \draw[kernel1] (left) to node [sloped,below] {\small } (root); ;
        \draw[kernel1,red] (leftl) to
     node [sloped,below] {\small }     (left);
     \draw[kernel1,green] (center) to
     node [sloped,below] {\small }     (root);
     \draw[kernel1,green] (center1) to
     node [sloped,below] {\small }     (root);
     \draw[kernel1,red] (leftlr) to
     node [sloped,below] {\small }     (leftl); 
     \draw[kernel1] (leftll) to
     node [sloped,below] {\small }     (leftl);
     \draw[kernel1,red] (leftr) to
     node [sloped,below] {\small }     (left);  
        \draw[kernel1] (right) to
     node [sloped,below] {\small }     (root);
      \draw[kernel1] (leftlc) to
     node [sloped,below] {\small }     (leftl); 
     \draw[kernel1,red] (rightr) to
     node [sloped,below] {\small }     (right);
     \draw[kernel1] (rightrr) to
     node [sloped,below] {\small }     (rightr);
     \draw[kernel1,red] (rightrl) to
     node [sloped,below] {\small }     (rightr);
     \draw[kernel1,red] (rightl) to
     node [sloped,below] {\small }     (right);
     \draw[kernel1,red] (rightlr) to
     node [sloped,below] {\small }     (rightl);
     \draw[kernel1] (rightll) to
     node [sloped,below] {\small }     (rightl);
     \end{tikzpicture}
      \longrightarrow   
 A = 
     \begin{tikzpicture}[scale=0.2,baseline=0.2cm]
        \node at (0,0)  [dot,red,label={[label distance=-0.2em]below: \scriptsize  $ \rho_{S_1} $}] (root) {};
         \node at (1,2)  [dot,red,label={[label distance=-0.2em]above: \scriptsize  $ \ell_4 $}] (right) {};
         \node at (-1,2)  [dot,red,label=left:] (left) {};
         \node at (0,4)  [dot,red,label={[label distance=-0.2em]above: \scriptsize  $ \ell_3 $} ] (leftr) {};
            \draw[kernel1,red] (right) to
     node [sloped,below] {\small }     (root); \draw[kernel1,red] (left) to
     node [sloped,below] {\small }     (root);
      \draw[kernel1,red] (leftr) to
     node [sloped,below] {\small }     (left);
     \end{tikzpicture} 
     \begin{tikzpicture}[scale=0.2,baseline=0.2cm]
        \node at (0,0)  [dot,red,label={[label distance=-0.2em]below: \scriptsize  $ \rho_{S_2} $}] (root) {};
         \node at (1,2)  [dot,red,label=right:] (right) {};
         \node at (-1,2)  [dot,red,label=left:] (left) {};
         \node at (-1,4)  [dot,red,label={[label distance=-0.2em]above: \scriptsize  $ \ell_8 $}] (leftr) {};
          \node at (1,4)  [dot,red,label={[label distance=-0.2em]above: \scriptsize  $ \ell_9 $} ] (rightr) {};
            \draw[kernel1,red] (right) to
     node [sloped,below] {\small }     (root); \draw[kernel1,red] (left) to
     node [sloped,below] {\small }     (root);
      \draw[kernel1,red] (leftr) to
     node [sloped,below] {\small }     (left);
     \draw[kernel1,red] (rightr) to
     node [sloped,below] {\small }     (right);
     \end{tikzpicture}
     \begin{tikzpicture}[scale=0.2,baseline=0.2cm]
        \node at (0,0)  [dot,green,label={[label distance=-0.2em]below: \scriptsize  $ \rho_{S_3} $}] (root) {};
         \node at (1,2)  [dot,green,label={[label distance=-0.2em]above: \scriptsize  $ \ell_6 $}] (right) {};
         \node at (-1,2)  [dot,green,label={[label distance=-0.2em]above: \scriptsize  $ \ell_5 $}] (left) {};
            \draw[kernel1,green] (right) to
     node [sloped,below] {\small }     (root); \draw[kernel1,green] (left) to
     node [sloped,below] {\small }     (root);
     \end{tikzpicture}.
     \end{equs}
     We have $ \lbrace S_3 \rbrace \in \Adm^r(T) $ and $ \lbrace  S_1 \,, S_2 \rbrace \in \Adm^{\circ}(\CR_{S_3}T) $. 
     
 In order to be able to prove the next proposition, we recall the definition of $ \Delta_2 $ given in \cite{BHZ}.  The map $ \Delta_2 :  \CT \rightarrow \CT \otimes \hat \CT_{+} $ is given for $ T^{\Labn}_{\Labe} \in \CT $ by: 
\begin{equs}\label{co-action}
 \Delta_2  T^{\Labn}_{\Labe} & = 
 \sum_{A \in  \Adm^+(T)} \sum_{\Labe_A,\Labn_A}  \frac1{\Labe_A!}
\binom{\Labn}{\Labn_A}
 (A,\Labn_A+\pi\Labe_A, \Labe \restr E_A) 
 \\& \qquad  \otimes   ( \CR_A T,2, [\Labn - \Labn_A]_A, \Labe + \Labe_A)\;, 
\end{equs}
where $ \Adm^+(T) $ is the set of subtrees $ A $ of $ T $ such that $ \rho_{A} = \rho_T $. Then one obtains $ \Deltap $ by applying $ \Pi_+ $: $ \Deltap = (\id \otimes \Pi_+ ) \Delta_2 $.

\begin{proposition} \label{commutation_R}
We have 
\begin{equs}
(\id \otimes \Deltap) \Deltam_r =  (\Deltam_r \otimes \id)\Deltap.
\end{equs} 
Moreover for all $ \ell \in \CG_{-} $, $ R_{\ell} $ commutes with $ G $.
\end{proposition}

\begin{proof}
The result follows from the identity: 
\begin{equs} \label{commutation_delta}
(\id \otimes \Delta_2) \Deltam_r =  (\Deltam_r \otimes \id)\Delta_2.
\end{equs}
Indeed, one has $ \Deltap = \left( \id \otimes \Pi_+  \right) \Delta_2 $ which gives:
\begin{equs}
\left( \id \otimes \Deltap \right) \Delta_r^- & = \left( \id \otimes \left( \id \otimes \Pi_+ \right) \Delta_2 \right) \Delta_r^- \\
& = \left( \Deltam_r \otimes \Pi_+  \right) \Delta_2 =  \left( \Deltam_r \otimes \id  \right) \Deltap.
\end{equs}
So it remains to prove \eqref{commutation_delta}. We first notice the next identity between $ \Deltam_r $ and $ \Delta_2 $:
\begin{equs} \label{e:deltar}
\Deltam_r = \left( \Pi_- \circ \tilde{\Pi} \otimes \mathcal{R}_{2} \right) \Delta_2,
\end{equs} 
where $ \mathcal{R}_{2} $ is the map which sends
$ (T,2)^{\Labn}_{\Labe} $ to $ T^{\Labn}_{\Labe} $ or
which removes the colour two at the root and $ \tilde{\Pi} : \CT \rightarrow \hat \CT_- $ is the projection which maps single node without decoration to the empty forest, any other tree is identified with the forest containing only this tree. The map $  \tilde{\Pi} $ is extended multiplicatively to the space $ \hat \CT_- $. Then \eqref{commutation_delta} follows from the co-associativity of $\Delta_2$ proved in \cite{BHZ}. Indeed by setting $ \tilde{\Pi}_- = \Pi_- \circ \tilde{\Pi}$, we have
\begin{equs}
( \id \otimes \Delta_2) \Deltam_r & = (\tilde{\Pi}_- \otimes \Delta_2 \mathcal{R}_2 ) \Delta_2 =  
(\tilde{\Pi}_- \otimes ( \mathcal{R}_2 \otimes \id ) \Delta_2  ) \Delta_2 \\ & = \left( (\tilde{\Pi}_- \otimes \mathcal{R}_{2}) \Delta_2 \otimes \id \right) \Delta_2 = (\Deltam_r \otimes \id ) \Delta_2. 
\end{equs}
Now for all $ \ell \in \CG_{-} $ and $g\in \CG_+$
\begin{equs}
R_\ell \Gamma_g=(\ell \otimes \id \otimes g )(\Deltam_r \otimes \id)\Deltap
=  (\ell \otimes \id \otimes g )(\id \otimes \Deltap)\Deltam_r =\Gamma_gR_\ell.
\end{equs}
\end{proof}

We finish this subsection by the proof of
the Theorem~\ref{inclusion theorem}:

\begin{proof}[Theorem~\ref{inclusion theorem}]
Let $ M \in \mathfrak{R}[\mathscr{T}] $, there exists $ \ell \in \CG_- $ such that
\begin{equs}
M = M_{\ell} = \left( \ell \otimes \id \right) \Deltam.
\end{equs}
From Corollary \ref{recursive_f}, we know that $ M_{\ell} = M^{\circ}_{\ell} R_{\ell} $ where
\begin{equs}
M^{\circ}_{\ell} = \left( \ell \otimes \id \right) \Deltam_{\circ}, \quad R_{\ell} = \left( \ell \otimes \id \right) \Deltam_{r}.
\end{equs} 
We need to check that $  R_{\ell} \in \mathcal{L}_{ad}(\CT) $ and that $ M^{\circ}_{\ell} $ satisfies the recursive definition \eqref{e:defM}. For $ M^{\circ}_{\ell} $, this property comes from Proposition \ref{mult_Mcirc} and from \eqref{recursive_identity_Mcirc} which gives for every $ X^n \CI_k^{\Labhom}(T^{\Labn}_{\Labe}) \in \CT $:
\begin{equs} \Deltam_{\circ}  X^n \CI_k^{\Labhom}(T^{\Labn}_{\Labe}) =( \id\otimes X^{\Labn}\I^{\Labhom}_{k})\Deltam  T^\Labn_\Labe.
\end{equs}
For $ R_{\ell} $, we first notice that for every
$ \tau \in \CT $, $\Deltam_{r} \tau $ is of the form
    $ \one_{1} \otimes \tau +  \sum_i \tau_i^{(1)} \otimes \tau^{(2)}_i  $ where 
    $  | \tau |_{\s} = | \tau_i^{(1)} |_{\s} + | \tau_i^{(2)} |_{\s} $ and $ || \tau || = || \tau_i^{(1)} || + || \tau_i^{(2)} || $, $ || \tau_i^{(2)} || <  || \tau ||$.
    The tree $ \tau_i^{(1)}  $ is of negative degree thus $  | \tau |_{\s} \geq | \tau_i^{(2)} |_{\s} $ and we get the following properties:$
    \Vert R_{\ell} \tau - \tau \Vert < \Vert \tau \Vert $ and $\vert R_{\ell} \tau - \tau \vert_{\s} < \vert \tau \vert_{\s} $.
    Let $ \CI_k^{\Labhom}(\tau) \in \hat B_{\circ} $, then by definition of $ \Adm^{r}(\CI_k^{\Labhom}(\tau)) $ we get $ \Adm^{r}(\CI_k^{\Labhom}(\tau)) \setminus (\hat B_{\circ} \cap \Adm^{r}(\CI_k^{\Labhom}(\tau)) ) = \emptyset $ which yields:
    \begin{equs}
    R_{\ell} \CI_k^{\Labhom}(\tau) =
    \left( \ell \otimes \id \right) \Deltam_{r} \CI_k^{\Labhom}(\tau) = \left( \ell \otimes \id \right) \left( \one_{1} \otimes  \CI_k^{\Labhom}(\tau) \right) = \CI_k^{\Labhom}(\tau).
    \end{equs}
  As the same any power of $ X $ at the root is killed by the character $ \ell $ which proves 
  the identity  $ R_{\ell} X^k \tau = X^k R_{\ell} \tau $ for every $ \tau \in \CT $.  
    The fact that $ R_{\ell} $ commutes with the structure group proceeds from Proposition~\ref{commutation_R}. From all these properties, we have $ M_{\ell} \in  \mathfrak{R}_{ad}[\mathscr{T}] $ which concludes the proof.
\end{proof}

\subsection{Recursive formula for the renormalisation group}

We want to derive a recursive formula for $ \Deltam $ by using the symbolic notation. The recursive formulation is based on the tree product and the inductive definition of the trees. The main difficulty at this point is that one can derive easily a recursive formula for $ \Deltap $ which is multiplicative for the tree product but not for $ \Deltam $. In order to recover this multiplicativity, we define a slight modification $ \hat \Delta_1 $ of $ \Deltam $ by adding more information. The main idea is to distinguish a tree in a forest.
This formulation allows to have a recursive definition for the two coproducts $ \Deltam $
and $ \Deltap $.

\begin{definition}
We define  $ \Tra_{\rho}$ the space of $ (F,\Labn,\Labe,\rho) = (F_{\Labe}^{\Labn},\rho) $ such that $ F $ is non- empty and such that $ \rho $ is the root of one connected component $ T_{\varrho} $ of $ F $. The space $ \Tra_{\rho}  $ can be viewed as the non-empty decorated forests where we distinguish one tree. We endow this space with a product $ \star $ defined as follows: 
\begin{equs}
(F,\Labn,\Labe,\rho) \star (\bar F,\bar \Labn  ,\bar \Labe, \bar \rho ) = 
(F_{\rho}^{c} \sqcup \bar F_{\bar \rho}^{c} \sqcup  \lbrace T_{\rho}\bar T_{\bar \rho} \rbrace, \Labn + \bar \Labn, \Labe + \bar \Labe , \tilde{\rho}).
\end{equs}
where $ F_{\rho}^{c} = F \setminus 
\lbrace T_{\rho} \rbrace $. We have also identified the two roots $ \rho $ and $ \bar \rho $ with $ \tilde{\rho} $ which means that
$ \Labn(\tilde{\rho}) = \Labn(\rho) + \bar \Labn(\bar \rho) $. We denote by $ \Trees_{\rho}  $ the forests $ (F^{\Labn}_{\Labe},\rho) \in \Tra_{\rho} $ such that $ F_{\rho}^{c}   $ is empty. Then we define a canonical
injection $ \iota_{\rho} : \Trees \rightarrow \Trees_{\rho} $, $ \iota_{\rho}(T_{\Labe}^{\Labn}) = ( T_{\Labe}^{\Labn}, \rho_T ) $ and we denote by $ \Pi^{\rho} $ its left inverse which is defined by:
\begin{equs}
\Pi^{\rho} :  \Tra_{\rho} \rightarrow  \Tra, \quad 
 \Pi^{\rho}(  (F,\Labn,\Labe,\rho) ) = (F,\Labn,\Labe).
\end{equs}
\end{definition}

\begin{definition} \label{adm_rho}
We define $ \Adm(F,\rho) $ as the family of all 
$ (A,\rho) \in \Tra_{\rho} $ such that $ A \in \Adm(F)$ and $ A $ contains all the nodes of the forest $ F $. 
\end{definition}

\begin{remark}
In the Definition~\ref{adm_rho}, we extract all the nodes
because we want to derive a recursive formula for $ \Deltam $ which means that one is not able to decide immediately during the recursive 
procedure in \eqref{e::recursive_delta1} if one node will belong to a tree extracted by $ \Deltam $.  
\end{remark}

We define $ \hat \Delta_1 : \Tra_{\rho} \rightarrow \Tra_{\rho} \otimes \Tra_{\rho}  $ in the following way:
\begin{equs}\label{def:coproduct_root}
& \hat \Delta_1 (F, \Labn,\Labe, \rho) = 
 \sum_{(A,\rho) \in \Adm(F,\rho)} \sum_{\Labe_A,\Labn_A} \frac1{\Labe_A!}
\binom{\Labn}{\Labn_A}
 (A,\Labn_A+\pi\Labe_A, \Labe \restr E_A, \rho) 
 \\& \qquad  \otimes( \CR_A F, [\Labn - \Labn_A]_A, \Labe + \Labe_A, \rho)\;. 
\end{equs}
The infinite sum makes sense as the same as in
\cite{BHZ} by using the bigraded structure of  $ \Tra_{\rho}$ and the fact that $ \hat \Delta_1 $ is a triangular map.

\begin{proposition} \label{proposition_coassociativity_delta1}
The map $ \hat \Delta_1 $ is multiplicative for $ \star $ and one has on $ \Tra_{\rho} $:
\begin{equs} \label{coassociativity_delta1}
( \id \otimes \hat \Delta_1   ) \hat\Delta_1  = ( \hat \Delta_1 \otimes \id   ) \hat\Delta_1
\end{equs}
\end{proposition}
\begin{proof}
The multiplicativity comes from the fact that 
\begin{equs}
\Adm(F,\rho) \star \Adm(\bar F,\bar \rho) = 
\Adm( (F,\rho) \star (F,\bar \rho) ).
\end{equs}
The rest of the proof can be seen as a consequence of the co-associativity results obtained in \cite{BHZ} see Remark~\ref{remarkproof}. In order to avoid the introduction of too much notations, we provide a direct proof in the appendix using the formula \eqref{e::recursive_delta1} with the symbolic notation.
\end{proof}

\begin{remark} \label{remarkproof}
One can define $ \Adm(F,\rho) $ using the formalism of the colours developed in \cite{BHZ} by $ \Adm_1(F,\hat F,\rho) $ where in this case $ \hat F^{-1}(1) = N_{F} $ and with the difference that we need to carry more information
by keeping track of the root of one tree in $ F $. This means that elements of $ \Tra_{\rho} $ have all their nodes coloured by the colour $ 1 $. Then for the definition of the coproduct, all the nodes are extracted.  We can then apply \cite[Prop. 3.9]{BHZ} without the extended decoration in order to recover the Proposition~\ref{proposition_coassociativity_delta1}.
\end{remark}

\begin{remark}
This coproduct $ \hat \Delta_1 $ contains at the same time the Connes-Kreimer coproduct and the extraction-contraction coproduct. In the sense that if we forget the root $ \rho $, we obtain a variant of the extraction-contraction coproduct where each node needs to be in one extracted subtree. On the other hand, if we quotient by the elements $ (F,\rho) $ such that $ F $ contains a tree with a root different from $ \rho $, we have the Connes-Kreimer coproduct. We make this statement more precise in Proposition~\ref{recursive_coproduct}.
\end{remark}

From  Corollary~\ref{def:coproduct_root}, we can derive a general recursive formula by using the symbolic notation. Before giving it, we need to encode $ (F^{\Labn}_{\Labe},\rho) $. For that, we introduce the new symbol $ \Co $ which is a map from $  \Tra_{\rho} $ into itself. Now $ (\lbrace T_{\rho}, \; T_1, \; ...., \;T_n \rbrace, \rho) $ is given by $ T_{\rho} \prod_{i=1}^n \Co(T_i) $ where $ T_{\rho} $ is the tree with root $ \rho $. Let $ (F^{\Labn}_{\Labe},\tilde{\rho}), (\bar F^{\bar \Labn}_{\bar \Labe},\bar \rho)  \in \Tra_{\rho} $, the product $ \star $ is given by:
\begin{equs}
(F^{\Labn}_{\Labe},\tilde{\rho}) \star (\bar F^{\bar \Labn}_{\Labe},\bar \rho) = T_{\tilde{\rho}} \bar T_{\bar \rho} \prod_{T \in  F_{\tilde{\rho}}^{c}} \Co(T) \prod_{\bar T \in \bar F_{\bar \rho}^{c}} \Co(\bar T). 
\end{equs}
Then the properties of the symbol $ \Co $ are:
\begin{enumerate}
\item The product $ \Co(T)\Co(\bar T) $ is associative and commutative as the product on the forests. Moreover, the map $ \Co $ is multiplicative for the product of forests, $ \Co(F \sqcup \bar F) = \Co(F ) \Co(\bar F) $.
\item The symbol $ \Co $ is also defined as an operator on $ (F^{\Labn}_{\Labe},\rho) $ in the sense that $ \Co((F^{\Labn}_{\Labe},\rho)) $ is the forest $ (F^{\Labn}_{\Labe} \sqcup \lbrace\bullet \rbrace,\bullet) $. Using only the symbolic notation, this can be expressed as:
\begin{equs}
\Co(T_{\rho} \Co( F_{\rho}^{c})) = \Co( T_{\rho}) \Co( F_{\rho}^{c}).
\end{equs}
\item The operator $ \CI^{\Labhom}_{k} $ is extended by acting only on  the tree with the root $ \rho $ in $ (F^{\Labn}_{\Labe},\rho) $:
\begin{equs}
\CI^{\Labhom}_{k}(T_{\rho} \Co(F_{\rho}^c)) = \CI^{\Labhom}_{k}(T_{\rho})\Co(F_{\rho}^c).
\end{equs}
\end{enumerate}

With these properties, we define the map $ \hat \Delta_1 $:
\begin{equs}[e::recursive_delta1]
&  \hat \Delta_1 \one =  \one \otimes \one ,\quad
  \hat \Delta_1  X_i = X_i  \otimes \one + \one \otimes X_i, 
     \\
     & \hat \Delta_1 \Co(\tau)  = \left( \Co \otimes  \Co \right) \hat \Delta_1 \tau,
      \\ &
\hat \Delta_1 \CI^{\Labhom}_k(\tau)  = \left( \CI^{\Labhom}_k \otimes \id
+ \sum_{\ell \in \N^{d+1}}  \frac{X^{\ell}}{\ell !} \Co \otimes  \CI^{\Labhom}_{k+ \ell}
\right) \hat \Delta_1 \tau, 
\end{equs}
where $ \one $ corresponds here to $ (\lbrace\bullet \rbrace,\bullet) $. We make an abuse of notations by identifying $ (T,\rho_T) $ and $ T $ when $ T $ is a tree.

 The previous recursive construction  and the properties of the map $ \Co $ can be explained graphically. We concentrate ourself on the shape and we omit the decorations. If we look at one term $ \tau_1 \otimes \tau_2 $ appearing in the decomposition of $  \hat \Delta_1  \tau $ for some tree $ \tau $ then $ \tau_1 $ is of the form $ \bar \tau  \prod_j \Co(\bar \tau_j)$. We colour trees of the form $  \Co(\bar \tau_j) $ in red and
 leave $ \bar \tau $ in black in the next example: 
     
 \begin{equs}
   (\CI_k^{\Labhom} \otimes \id + \Co \otimes \CI^{\Labhom}_k)   \left(  
   \begin{tikzpicture}[scale=0.2,baseline=0.2cm]
        \node at (0,0)  [dot,label=below:] (root) {};  
     \end{tikzpicture}
     \begin{tikzpicture}[scale=0.2,baseline=0.2cm]
        \node at (0,0)  [dot,red,label=below:] (rootbis) {};
        \node at (2,2)  [dot,red,label=below:] (rightbis) {};
        \node at (0,2)  [dot,red,label=above:] (root) {};
         \node at (0,4)  [dot,red,label=above:] (center) {};
         \node at (-2,4)  [dot,red,label=above:] (left) {};
         \node at (2,4)  [dot,red,label=above:  ] (right) {};
         \draw[kernel1,red] (root) to
     node [sloped,below] {\small }               (rootbis);
            \draw[kernel1,red] (right) to
     node [sloped,below] {\small }     (root); \draw[kernel1,red] (left) to
     node [sloped,below] {\small }     (root);
     \draw[kernel1,red] (center) to
     node [sloped,below] {\small }     (root);
     \draw[kernel1,red] (rightbis) to
     node [sloped,below] {\small }     (rootbis);
     \end{tikzpicture} 
     \begin{tikzpicture}[scale=0.2,baseline=0.2cm]
        \node at (0,0)  [dot,red,label=below:] (root) {};
         \node at (2,2)  [dot,red,label=right:] (right) {};
         \node at (-2,2)  [dot,red,label=left:] (left) {};
         \node at (-3,4)  [dot,red,label=above:  ] (leftl) {};
          \node at (-1,4)  [dot,red,label=above: ] (leftr) {};
          \node at (1,4)  [dot,red,label=above:  ] (rightl) {};
           \node at (3,4)  [dot,red,label=above:  ] (rightr) {};
            \draw[kernel1,red] (right) to
     node [sloped,below] {\small }     (root); \draw[kernel1,red] (left) to
     node [sloped,below] {\small }     (root);
     \draw[kernel1,red] (leftr) to
     node [sloped,below] {\small }     (left);
      \draw[kernel1,red] (rightr) to
     node [sloped,below] {\small }     (right);
      \draw[kernel1,red] (rightl) to
     node [sloped,below] {\small }     (right);
      \draw[kernel1,red] (leftl) to
     node [sloped,below] {\small }     (left);
     \end{tikzpicture} 
       \otimes \quad \tau_2 \right) \quad =  
 \\    
  \begin{tikzpicture}[scale=0.2,baseline=0.2cm]
     \node at (0,0)  [dot,label=below:] (rootbis) {};
        \node at (0,2)  [dot,label=below:] (root) {};
         \draw[kernel1] (root) to node [sloped,below] {\small } (rootbis);  
     \end{tikzpicture}   \begin{tikzpicture}[scale=0.2,baseline=0.2cm]
        \node at (0,0)  [dot,red,label=below:] (rootbis) {};
        \node at (2,2)  [dot,red,label=below:] (rightbis) {};
        \node at (0,2)  [dot,red,label=above:] (root) {};
         \node at (0,4)  [dot,red,label=above:] (center) {};
         \node at (-2,4)  [dot,red,label=above:] (left) {};
         \node at (2,4)  [dot,red,label=above:  ] (right) {};
         \draw[kernel1,red] (root) to
     node [sloped,below] {\small }               (rootbis);
            \draw[kernel1,red] (right) to
     node [sloped,below] {\small }     (root); \draw[kernel1,red] (left) to
     node [sloped,below] {\small }     (root);
     \draw[kernel1,red] (center) to
     node [sloped,below] {\small }     (root);
     \draw[kernel1,red] (rightbis) to
     node [sloped,below] {\small }     (rootbis);
     \end{tikzpicture} 
     \begin{tikzpicture}[scale=0.2,baseline=0.2cm]
        \node at (0,0)  [dot,red,label=below:] (root) {};
         \node at (2,2)  [dot,red,label=right:] (right) {};
         \node at (-2,2)  [dot,red,label=left:] (left) {};
         \node at (-3,4)  [dot,red,label=above:  ] (leftl) {};
          \node at (-1,4)  [dot,red,label=above: ] (leftr) {};
          \node at (1,4)  [dot,red,label=above:  ] (rightl) {};
           \node at (3,4)  [dot,red,label=above:  ] (rightr) {};
            \draw[kernel1,red] (right) to
     node [sloped,below] {\small }     (root); \draw[kernel1,red] (left) to
     node [sloped,below] {\small }     (root);
     \draw[kernel1,red] (leftr) to
     node [sloped,below] {\small }     (left);
      \draw[kernel1,red] (rightr) to
     node [sloped,below] {\small }     (right);
      \draw[kernel1,red] (rightl) to
     node [sloped,below] {\small }     (right);
      \draw[kernel1,red] (leftl) to
     node [sloped,below] {\small }     (left);
     \end{tikzpicture} 
       \otimes \quad \tau_2 \quad + \begin{tikzpicture}[scale=0.2,baseline=0.2cm]
        \node at (0,0)  [dot,label=below:] (root) {};  
     \end{tikzpicture}
      \begin{tikzpicture}[scale=0.2,baseline=0.2cm]
        \node at (0,0)  [dot,red,label=below:] (root) {};  
     \end{tikzpicture}
 \begin{tikzpicture}[scale=0.2,baseline=0.2cm]
        \node at (0,0)  [dot,red,label=below:] (rootbis) {};
        \node at (2,2)  [dot,red,label=below:] (rightbis) {};
        \node at (0,2)  [dot,red,label=above:] (root) {};
         \node at (0,4)  [dot,red,label=above:] (center) {};
         \node at (-2,4)  [dot,red,label=above:] (left) {};
         \node at (2,4)  [dot,red,label=above:  ] (right) {};
         \draw[kernel1,red] (root) to
     node [sloped,below] {\small }               (rootbis);
            \draw[kernel1,red] (right) to
     node [sloped,below] {\small }     (root); \draw[kernel1,red] (left) to
     node [sloped,below] {\small }     (root);
     \draw[kernel1,red] (center) to
     node [sloped,below] {\small }     (root);
     \draw[kernel1,red] (rightbis) to
     node [sloped,below] {\small }     (rootbis);
     \end{tikzpicture} 
     \begin{tikzpicture}[scale=0.2,baseline=0.2cm]
        \node at (0,0)  [dot,red,label=below:] (root) {};
         \node at (2,2)  [dot,red,label=right:] (right) {};
         \node at (-2,2)  [dot,red,label=left:] (left) {};
         \node at (-3,4)  [dot,red,label=above:  ] (leftl) {};
          \node at (-1,4)  [dot,red,label=above: ] (leftr) {};
          \node at (1,4)  [dot,red,label=above:  ] (rightl) {};
           \node at (3,4)  [dot,red,label=above:  ] (rightr) {};
            \draw[kernel1,red] (right) to
     node [sloped,below] {\small }     (root); \draw[kernel1,red] (left) to
     node [sloped,below] {\small }     (root);
     \draw[kernel1,red] (leftr) to
     node [sloped,below] {\small }     (left);
      \draw[kernel1,red] (rightr) to
     node [sloped,below] {\small }     (right);
      \draw[kernel1,red] (rightl) to
     node [sloped,below] {\small }     (right);
      \draw[kernel1,red] (leftl) to
     node [sloped,below] {\small }     (left);
     \end{tikzpicture} 
  \otimes \quad \CI^{\Labhom}_k(\tau_2).  
     \end{equs}

In the next proposition, we prove the equivalence between the recursive definition on the symbols and the definition   \eqref{def:coproduct_root} . 
\begin{proposition}
The definitions  \eqref{def:coproduct_root}  and  \eqref{e::recursive_delta1}  coincide.
\end{proposition}
\begin{proof} The fact that $\hat \Delta_1 \one = \one \otimes \one $
follows immediately from the definitions. The element $X^k $ is encoded by the tree 
consisting of just a root, but with label $k$: $ ( \lbrace \bullet^k \rbrace,\bullet)  $. One then has $A = \lbrace \bullet \rbrace$
and $\Labe_A = 0$, while $\Labn_A$ runs over all possible decorations for the root.
This shows that \eqref{def:coproduct_root}  yields in this case
\begin{equ}
\hat \Delta_1 X^k = \sum_{\ell \leq k} \binom{k}{\ell} X^{k-\ell} \otimes X^\ell\;,
\end{equ}
which is as required.
It now remains to verify that the recursive identities hold as well. The coproduct $ \hat \Delta_1 $ is multiplicative for $ \star $. Thus we can restrict ourselves to $\hat \Delta_1 \Co(T^{\Labn}_\Labe) $ and $ \hat \Delta_1 \CI_k^{\Labhom}(T^{\Labn}_\Labe) $ where $ T^{\Labn}_{\Labe} $ is a decorated tree.
For the first, we just notice that:
\begin{equs} 
\hat \Delta_1 (\lbrace T^{\Labn}_{\Labe}  \rbrace \sqcup \lbrace \rho \rbrace,\rho) = \left( (\id,\rho) \otimes (\id,\rho) \right) \hat \Delta_1 (\lbrace T^{\Labn}_{\Labe} \rbrace,\rho_{T})
\end{equs}
where $ \rho $ is different from the root $ \rho_T $ of $ T $  and $ (\id,\rho) $ replaces the root by $ \rho $ and adds $ \rho $ to the forest. The operator $ (\id, \rho) $ is identified with $ \mathcal{C} $.
It remains to consider $\hat \Delta_1 \CI^{\Labhom}_k(T^{\Labn}_\Labe)$.
It follows from the definitions that by denoting $ \rho $ the root of $ \CI^{\Labhom}(T) $  
\begin{equs}
\Adm(\CI^{\Labhom} (T),\rho) =\CI^{\Labhom}( \Adm(T,\rho_T)) \sqcup \lbrace ( A \sqcup \lbrace \rho \rbrace,\rho) : (A,\rho_T) \in \Adm(T,\rho_T) \rbrace.
\end{equs}
This decomposition translates the belonging or not of the
edge $e_{\rho} = \CI^{\Labhom}$ to $ (A,\rho) \in \Adm(\CI^{\Labhom}(T),\rho) $.
Given $ (A,\rho_{T}) \in \Adm(T,\rho_T)$, since the root-decoration of $\CI_k^{\Labhom}(T)$ is $0$, the set of 
all possible node-labels $\Labn_A$ for $\CI_k^{\Labhom}(T)$ appearing in \eqref{def:coproduct_root} for 
$\hat \Delta_1 \CI_k^{\Labhom}(T^{\Labn}_\Labe)$ coincides with those appearing in the expression for 
$\hat \Delta_1 T^{\Labn}_\Labe$. Furthermore, it follows from the definitions that for any such $A$ one has
\begin{equ}
\CR_A \CI^{\Labhom}(T)
=  \CI^{\Labhom}(\CR_A T) \;,
\end{equ}
so that we have the identity
\begin{equs}
\hat \Delta_1 \CI_k^{\Labhom}(T^{\Labn}_\Labe) &= (\CI^{\Labhom}_k \otimes \id) \hat \Delta_1 T^{\Labn}_\Labe
+ \sum_{\Labe_{\rho},\Labn_{\rho}} {1\over \Labe_{\rho}!}\binom{\Labn }{\Labn_{\rho}}
(\lbrace \rho \rbrace,\Labn_{\rho} + \pi \Labe_{\rho},\rho ) \star (\id,\rho) \\
& \qquad \otimes (\lbrace \rho \rbrace, \Labn(\rho) -  \Labn_{\rho},\rho) \star \CI^{\Labhom}_{k+\Labe_{\rho}(e_\rho)} \;  \hat \Delta_1 T_{\Labe}^{\Labn} \\ 
& = (\CI^{\Labhom}_k \otimes \id) \hat \Delta_1 T^{\Labn}_\Labe
+ \left( \sum_{\Labe_{\rho}} {1\over \Labe_{\rho}!} 
X^{ \pi \Labe_{\rho} } \Co \otimes 
  \CI^{\Labhom}_{k+\Labe_{\rho}(e_\rho)} \right) \hat \Delta_1 T^{\Labn}_{\Labe}  \;,
\end{equs}
because $ \Labn(\rho) =0 $ so that
$\Labn_{\rho}$ is a zero ( $\Labn(\rho) -  \Labn_{\rho}  \geq 0 $  ) . 
Note now that $\Labe_{\rho}$  consists of a single decoration 
(say $\ell$), supported on $e_\rho$.  As a consequence, we can rewrite the above as
\begin{equs}
\hat \Delta_1 \CI_k^{\Labhom}(T^{\Labn }_\Labe) = \left(\CI_k^{\Labhom} \otimes \id
+ \sum_{\ell \in \N^{d+1}} {X^{\ell}  \over \ell !}  \Co \otimes  \CI_{k+\ell}^{\Labhom} \right)\hat \Delta_1 T^{\Labn}_\Labe.
\end{equs}
\end{proof}

\begin{remark}
We can derive a recursive formula for $ \hat \Delta_1 $ with the extended decorations introduced in \cite{BHZ}: 
\begin{equs}
&  \hat \Delta_1 \one_{ \Labo} =  \one_{ \Labo} \otimes \one_{ \Labo} ,\quad
  \hat \Delta_1  X_i = X_i  \otimes \one_{e_i} + \one \otimes X_i, 
     \\
     & \hat \Delta_1 \Co(\tau)  = \left( \Co \otimes  \Co \right) \hat \Delta_1 \tau,
      \\ &
\hat \Delta_1 \CI^{\Labhom}_k(\tau)  = \left( \CI^{\Labhom}_k \otimes \one_{\Labhom-k}
+ \sum_{\ell \in \N^{d+1}}  \frac{X^{\ell}}{\ell !} \Co \otimes \one_{\ell}  \CI^{\Labhom}_{k+ \ell}
\right) \hat \Delta_1 \tau, 
\end{equs}
where $ \one_{\Labo} $ is the tree with one node 
having the extended decoration $ \Labo \in \Z^{d+1}
\oplus \Z(\mfL)  $.
\end{remark}

Before making the link between $ \hat \Delta_1 $
and the maps $ \Deltam $, $ \Deltap $, we introduce some notations. We denote by 
$ \mathcal{C}_2  $ the map acting on trees by sending $ T^{\Labn}_{\Labe} $ to $ (T,2)^{\Labn}_{\Labe} $ or equivalently by colouring the root with the colour $ 2 $. Then we define the map $ \Pi_{\Trees} : \Forests_{\rho} \rightarrow \Trees_{\rho} $ by sending  $ (F,\rho) $ to $ (T_{\rho},\rho) $ when $ F = T_{\rho} \cdot \bullet \cdot ... \cdot \bullet $ and zero otherwise. In the next proposition, we also use the map 
$ \tilde{\Pi} : \hat \CT_- \rightarrow \hat \CT_-$ defined for the identity \eqref{e:deltar}. This map allows to remove the isolated nodes.

 \begin{proposition} \label{recursive_coproduct}
    One has on $ \CT $: 
    \begin{equs}
   \Deltam  =  (\Pi_- \circ \tilde \Pi \circ \Pi^{\rho}  \otimes \Pi^{\rho}) \hat \Delta_1 \iota_{\rho}, \quad 
  \Deltap  = (  \Pi^{\rho} \circ \Pi_{\Trees}  \otimes  \Pi_+ \circ \mathcal{C}_2  \circ \Pi^{\rho}) \hat \Delta_1 \iota_{\rho} .
     \end{equs} 
    \end{proposition}
    \begin{proof}
    The proof follows from the expressions of 
    $ \Deltam$, $ \Deltap $ and $ \hat \Delta_1 $ respectively given in \eqref{co-action minus}, \eqref{co-action} and \eqref{def:coproduct_root}. Indeed, we have a bijection between elements of $ \Adm(F,\rho) $ which are of the form
    $ (T_{\rho} \cdot \bullet \cdot ... \cdot \bullet, \rho) $ and $ \Adm^+(T) $. We also have a bijection between the $ A\in  \Adm(T) $ which don't have any isolated nodes and $  \Adm(F,\rho) $.
    These bijections come from the fact that all the nodes
    are extracted in the definition of $ \Adm(F,\rho) $.
    \end{proof}

\section{Examples of renormalised models}
\label{section 5}
For the examples of this section, we define the renormalisation map $ M_{R} $ by using an admissible map $  R_{\ell} = (\ell \otimes \id) \Deltam_r $ with $ \ell \in \CG_- $.
Moreover for each example, we describe the structure and we look at the following properties which a model could verify or not:
\begin{enumerate}[label=($\alph*$)]
\item \label{prop1} The map $ M $ commutes with $ G $ .
\item \label{prop2} For every symbol $ \tau $, $ \Pi_x^M \tau = \Pi_x M \tau $.
\item \label{prop3}For every symbol $ \tau $, $ (\Pi_x^M \tau ) (x) = (\Pi_x M \tau) (x) $.
\end{enumerate}

\begin{remark}
In this section, we will give examples which do not verify the first two properties. But the last one is verified by all the examples. In the framework of the extended structure, we directly have the second property see \cite[Thm 6.15]{BHZ}.
\end{remark}

We start with a toy model on the Wick renormalisation then we move on to examples in singular SPDEs.

\subsection{Hermite polynomials}
\label{hermite_polynomials}

We look at a very simple example: the powers of a standard Gaussian random variable $ \xi $ with zero mean and covariance $ c^2$, which can be interpreted as a white noise on a singleton $\{x\}$. The space $ \mathcal{T} $ is given as the linear span of $ \lbrace \Xi^n: n \in \N \rbrace $ and $G = \lbrace  \id \rbrace $ . 
 Given the natural definition
 \[
 \PPi \, \Xi^n=\xi^n,
 \]
 we want to find $M$ such that the renormalised $n$-th power of $ \xi $ is the Wick product: 
 \[
 \PPi^M \Xi^n = \xi^{\diamond n} = H_n(\xi,c)
 \]
where $H_n$ are generalised Hermite polynomials: $H_0=1$, $ H_{n+1}(x,c) = x H_{n}(x,c) -c^2 H'_{n}(x,c) $.
One natural way of defining $ M $ is $ M = \exp(-R_{\ell}) $ where $ R_{\ell} = (\ell \otimes \id) \Deltam_r $ and $ \ell(\Xi^n)=c^2 \un{(n=2)} $. The map $ \Deltam_r $ is defined on $ \Xi^n $   by: 
\[
 \Deltam_r \Xi^n = \sum_{k = 0}^n \binom{n}{k}
 \Xi^k \otimes \Xi^{n-k}.
\]
This map can be expressed in our general setting with the same subset $ \Adm^r(T) $. But in that case, we do not have any decorations. In the next example, we will encode $ \Xi^n $ with a set of $ n  $ leaves. A rooted subtree is identified with a subset of leaves. We present one term in the decomposition of $ \Deltam_r \Xi^8 $ which is in the support of $ \ell $:
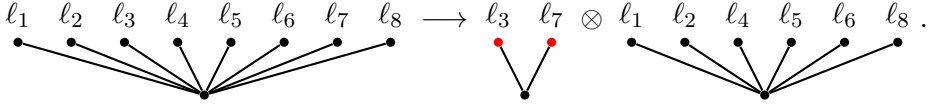
\begin{equs}
\begin{tikzpicture}[scale=0.35,baseline=0.2cm]
       \node at (1,-2)  [dot,label=above:$  $] (root) {};
        \node at (0,0)  [dot,label=above:$ \ell_4 $] (root1) {};
        \node at (-2,0)  [dot,label=above:$ \ell_3 $] (root2) {};
        \node at (-4,0)  [dot,label=above:$ \ell_2 $] (root3) {};
        \node at (-6,0)  [dot,label=above:$ \ell_1 $] (root4) {};
        \node at (2,0)  [dot,label=above:$ \ell_5 $] (root5) {};
        \node at (4,0)  [dot,label=above:$ \ell_6 $] (root6) {};
        \node at (6,0)  [dot,label=above:$ \ell_7 $] (root7) {};
         \node at (8,0)  [dot,label=above:$ \ell_8 $] (root8) {};
   \draw[kernel1] (root1) to
     node [sloped,below] {\small }     (root);      
     \draw[kernel1] (root2) to
     node [sloped,below] {\small }     (root);      
     \draw[kernel1] (root3) to
     node [sloped,below] {\small }     (root);      
     \draw[kernel1] (root4) to
     node [sloped,below] {\small }     (root);      
      \draw[kernel1] (root5) to
     node [sloped,below] {\small }     (root);      
     \draw[kernel1] (root6) to
     node [sloped,below] {\small }     (root);      
     \draw[kernel1] (root7) to
     node [sloped,below] {\small }     (root);      
     \draw[kernel1] (root8) to
     node [sloped,below] {\small }     (root);      
     \end{tikzpicture} 
   \longrightarrow    \begin{tikzpicture}[scale=0.35,baseline=0.2cm]
    \node at (1,-2)  [dot,label=above:$  $] (root) {};
    \node at (0,0)  [dot,red,label=above:$ \ell_3 $] (root1) {};
   \node at (2,0)  [dot,red,label=above:$ \ell_7 $] (root2) {};      
    \draw[kernel1] (root1) to
     node [sloped,below] {\small }     (root);      
     \draw[kernel1] (root2) to
     node [sloped,below] {\small }     (root);      
     \end{tikzpicture} 
   \otimes 
     \begin{tikzpicture}[scale=0.35,baseline=0.2cm]
     \node at (1,-2)  [dot,label=above:$  $] (root) {};
        \node at (0,0)  [dot,label=above:$ \ell_4 $] (root1) {};
        \node at (-2,0)  [dot,label=above:$ \ell_2 $] (root2) {};
        \node at (-4,0)  [dot,label=above:$ \ell_1 $] (root3) {};
        \node at (2,0)  [dot,label=above:$ \ell_5 $] (root4) {};
        \node at (4,0)  [dot,label=above:$ \ell_6 $] (root5) {};
         \node at (6,0)  [dot,label=above:$ \ell_8 $] (root6) {};
          \draw[kernel1] (root1) to
     node [sloped,below] {\small }     (root);      
     \draw[kernel1] (root2) to
     node [sloped,below] {\small }     (root);      
     \draw[kernel1] (root3) to
     node [sloped,below] {\small }     (root);      
     \draw[kernel1] (root4) to
     node [sloped,below] {\small }     (root);      
      \draw[kernel1] (root5) to
     node [sloped,below] {\small }     (root);      
     \draw[kernel1] (root6) to
     node [sloped,below] {\small }     (root);      
     \end{tikzpicture}. 
     \end{equs}
     We have removed the set $ \lbrace \ell_{3}, \ell_{7} \rbrace $ from the term $ \Xi^8 $.
  Then $ \PPi^M $ is given by: 
\[
 \PPi^M \Xi^n = \PPi M \Xi^n.
\]

By definition,  this example verifies all the three properties \ref{prop1}, \ref{prop2} and \ref{prop3}.
 We are able to provide a description of $ M $: 
 \begin{proposition}
 The map $ M $ is given by: 
\begin{equs}
M = M_{\ell_{wick}} =  ( \ell_{wick} \otimes \id ) \Deltam_r
\end{equs}
where
\begin{equs}
\ell_{wick} = \one^{*} + \sum_{k \geq 1} (-1)^{k}   
\frac{(2k-1)!}{2^{k-1} (k-1)!}  
   c^{2k} \mathds{1}_{\lbrace \Xi^{2k} \rbrace}.
\end{equs} 
 \end{proposition}
 \begin{proof}
  We use the following lemma:
  \begin{lemma}
  \label{hermite}
  For every $ k \in \N^* $ 
  $
   \frac{R^k}{k!} = ( f_k \otimes \id) \Deltam_r  
  $, where 
  \[
   f_k = \frac{(2k-1)!}{2^{k-1} (k-1)!}  
   c^{2k} \mathds{1}_{\lbrace \Xi^{2k} \rbrace}  .
  \]
  \end{lemma}
  \begin{proof}
  We proceed by recurrence. It is obvious for $ k =1 $. Let $ k \in \N^* $, we suppose the property true for this integer.  We have for every $ n \in \N $ 
  \[
   \begin{aligned}
   \frac{R^{k+1}}{(k+1)!} \Xi^n & = 
   \frac{1}{k+1} R (f_k \otimes \id) \Deltam_r \Xi^n \\
   & = \frac{1}{k+1}  (f_k \otimes \ell \otimes \id) (\id \otimes \Deltam_r) \sum_{m=0}^n \binom{n}{m} \Xi^{m} \otimes \Xi^{n-m} \\
   & = \frac{1}{k+1}  (f_k \otimes \ell \otimes \id)  \sum_{m=0}^n \sum_{\ell = 0}^{n-m} \binom{n}{m} \binom{n-m}{\ell} \Xi^{m} \otimes \Xi^{l} \otimes \Xi^{n-m-\ell} \\
   & = \mathds{1}_{\lbrace n \geq 2k + 2 \rbrace} \frac{1}{k+1}   \frac{(2k-1)!}{2^{k-1} (k-1)!}  
   c^{2k+2} \binom{n}{2k} \binom{n-2k}{2}  \\
   & = \mathds{1}_{\lbrace n \geq 2k + 2 \rbrace} \frac{(2k+1)!}{2^{k} k!}  
   c^{2k+2}. 
   \end{aligned}
  \]
  \end{proof}
  From the previous lemma, we deduce that $ \ell_{wick} = \one^{*} +  \sum_{k \geq 1} (-1)^k f_k $ which concludes the proof.
 \end{proof}

\begin{remark}
In this particular case, we look at a recursive formulation of $ M = M^{\circ} R $, $ R $ turns out to be equal to $ M $ because $ M^{\circ} $ is the identity on $ \CT $. This is also the reason why 
$ \ell_{wick} $ has a complicate expression in comparison to $ \ell $.
\end{remark}

\subsection{The KPZ equation }
The KPZ equation is given on $ \R $ by
\begin{equs}
\partial_t u = \partial_x^2 u + ( \partial_x u )^2 + \xi
\end{equs}
where $ \xi $ is the space-time white noise. In this equation, there is only one noise denoted by $ \Xi $ in the symbolic notation. Moreover, we denote by $ \CI $ the abstract integrator associated to the heat kernel $ K $. The scaling $ \s $ is the parabolic scaling $ (2,1) $. In the sequel, we make the following abuse of notation $ \CI_{(0,1)} = \CI_1 $.
The renormalisation group for the KPZ equation has been introduced in \cite{KPZ} and it has been given in the setting of regularity structure in \cite{reg}.
We first present the normal and complete rule $ \mathcal{R}_{kpz} $ used for building $ \mathcal{T}_{kpz} $ following \cite[Sec. 5.4]{BHZ}: 
\begin{equs}
 \mathcal{R}_{kpz}(\Xi) = \lbrace () \rbrace, \quad   \mathcal{R}_{kpz}(\CI) = \lbrace (),(\Xi), (\CI_1), (\CI_1,\CI_1) \rbrace.
\end{equs} 
The admissible map $ R = R_{kpz} $ associated to the map $ M $ is given by $ R_{kpz} = ( \ell_{kpz} \otimes \id) \Deltam_r $ where $ \ell_{kpz} $ is non zero for: 
\begin{equs}     
 \tau \in \lbrace \CI_1( \Xi)^2,  \mathcal{I}_{1}(\mathcal{I}_{1}(\Xi)^2)^2, \mathcal{I}_{1}(\Xi)
 \mathcal{I}_{1}( \mathcal{I}_{1}(\Xi) \mathcal{I}_{1}( \mathcal{I}_{1}(\Xi)^2)) \rbrace.
\end{equs}

For the next Proposition~\ref{property KPZ}, we need to work with a different rule:
\begin{equs}
\bar{\mathcal{R}}_{kpz}(\Xi) = \lbrace () \rbrace, \quad \bar{\mathcal{R}}_{kpz}(\CI) = \lbrace (\Xi), (\CI_1), (\CI_1,\CI_1) \rbrace.
\end{equs}
We consider a new space $ \bar \CT $ generated by this rule such that the symbol $ \CI(\tau) $ is zero whenever $ \tau $ is a polynomial.
This space was used in $ \cite{reg} $.
In practice for the expansion of the solution, we work with a truncated version of it by considering only the symbols below a certain degree. It is the same for $ \bar \CT_{+} $. Then only a finite number of polynomials appears when we apply the structure group on $ \bar \CT $.
We make the same assumption as in \cite{reg} that the kernel $ K $ associated to $ \CI $ annihilates polynomials up to a certain order bigger than the order of the truncation.

\begin{proposition} \label{property KPZ}
The map $ M = M_{R_{kpz}} $ satisfies the properties \ref{prop1}, \ref{prop2} and \ref{prop3}.
\end{proposition}

\begin{proof}
For proving the first two properties \ref{prop1} and \ref{prop2}, we need the following lemma: 
\begin{lemma}
For every symbol $ \tau $, $ M^{\circ} \tau = \tau $ and there exists a polynomial $ P_{\tau} $ such that: $ M \tau = \tau + P_{\tau}(X)$.
\end{lemma}

\begin{proof}
We proceed by induction.  It is obvious for $ X_i $ and $ \Xi $. For $ \tau = \mathcal{I}_1(\tau') $, we apply the induction hypothesis on  $ \tau'  $ which gives
\[
 M^{\circ} \CI_1(\tau') = \CI_1(M \tau')= \CI_1(\tau' + P_{\tau'}(X)) = \CI_1(\tau').
\] 
Let $ \tau = \prod_i \tau_i $ a product of elementary symbols. From the induction hypothesis on the $ \tau_i $, it follows 
\begin{equs}
M \prod_i \tau_i = M^{\circ} (R \tau - \tau) + \prod_i M^{\circ} \tau_i 
 = M^{\circ} (R \tau - \tau) + \tau. 
\end{equs}
Then $ R \tau - \tau $ is non zero if one element in the support of $ \ell_{kpz} $ is a subtree of $ \tau $. Necessarily, $ \tau $ should be of the form
\begin{itemize}
\item $ \tau_1\CI_1(\Xi \tau_2)\CI_1(\Xi \tau_3)    $,
\item 
$ \tau_1\mathcal{I}_{1}(\Xi \tau_2)
 \mathcal{I}_{1}(\tau_3 \mathcal{I}_{1}(\Xi \tau_4) \mathcal{I}_{1}( \tau_5 \mathcal{I}_{1}(\tau_6 \Xi)\mathcal{I}_{1}(\tau_7 \Xi))) $,
 \item $ \tau_1 \mathcal{I}_{1}( \tau_2 \mathcal{I}_{1}( \tau_3 \Xi)\mathcal{I}_{1}( \tau_4 \Xi)) \mathcal{I}_{1}( \tau_5 \mathcal{I}_{1}(\tau_6 \Xi)\mathcal{I}_{1}(\tau_7 \Xi)) $,
 \end{itemize} where the $ \tau_i  $ belong to $ \mathcal{T}_{kpz} $. By looking, at the rule available in $ \mathcal{R}_{kpz} $, we deduce that the $ \tau_i $ should be monomials of the form $ X^k $. Therefore, $ R \tau - \tau $ is a polynomial which allows us to conclude.
\end{proof}
For the property \ref{prop1}, we proceed by induction and we also prove that $ {M^{\circ}}  $ commutes with $ G $ . Let $ \Gamma \in G $, the proof is obvious for $ X_i $ and $ \Xi $. 
 Let $ \tau = \CI_1(\tau') $, it happens
\begin{equs}
M \Gamma \CI_1(\tau')  & =
  M \left( \Gamma \CI_1(\tau') -  \CI_1( \Gamma \tau') + \CI_1( \Gamma \tau') \right) 
  \\ & = \Gamma \CI_1(\tau') -  \CI_1( \Gamma \tau') + \CI_1(\Gamma M \tau').
\end{equs} 
On the other hand, we have
\[
\begin{aligned}
\Gamma M  \CI_1(\tau') & =
  \Gamma \CI_1(M \tau') -  \CI_1( \Gamma M  \tau') + \CI_1( \Gamma  M \tau') .
\end{aligned}
\] 
Using the previous lemma, it follows $  \CI_1(M \tau') =  \CI_1( \tau') $ and $  \CI_1( \Gamma M \tau' ) =  \CI_1(\Gamma \tau') $ which give the result. The same proof works for $ M^{\circ} $.

Let $ \tau = \prod_i \tau_i $ a product of elementary symbols, we have
\begin{equs}
 M \Gamma \tau & = 
  M^{\circ} R \Gamma \tau = M^{\circ}  \Gamma R \tau = M^{\circ}  \Gamma (R \tau-\tau) + \prod_i M^{\circ}  \Gamma \tau_i  \\
 & =   \Gamma M^{\circ} (R \tau-\tau) + \prod_i  \Gamma M^{\circ} \tau_i = \Gamma M \tau.
\end{equs} 
where we have used the fact that $ M  $ commutes with $ \Gamma $ on the $ \tau_i $ and
$ R \tau $ which comes from the induction hypothesis and the fact that $ R $ commutes with $ G $ .
For the property \ref{prop2}, we proceed as the same by induction. The only difficult point is for $ \tau = \CI_1(\tau') $. We have by applying the induction hypothesis, the previous lemma on $ \tau' $ and using the fact that $ K $ integrates to zero against polynomials
\begin{equs}
 (\Pi_{x}^{M} \CI_{1} \tau')(y)  & = 
  \int D^{(0,1)} K(y-z) (\Pi^{M}_{x}
  \tau')(z) dz \\ & - \sum_{\ell} 
  \frac{(y-x)^{\ell}}{\ell!} \mathds{1}_{| \tau' |_{\s} + 1 -\ell > 0} \int D^{(0,1)+\ell} K(x-z) (\Pi^{M}_{x}
  \tau')(z) dz \\
  &   = 
  \int D^{(0,1)} K(y-z) (\Pi_{x}
  \tau')(z) dz \\ & - \sum_{\ell} 
  \frac{(y-x)^{\ell}}{\ell!} \mathds{1}_{| \tau' |_{\s} + 1 -\ell > 0} \int D^{(0,1)+\ell} K(x-z) (\Pi_{x}
  \tau')(z) dz.
  \end{equs}
which allow us to conclude.
\end{proof}

\subsection{The generalised KPZ}
\label{generalised_KPZ_ex}
The equation contains the previous equation and it is given by: 
\[
  \partial_{t} u = \partial_{x}^{2} u + g(u) \left( \partial_{x} u \right)^{2} + h(u) \partial_{x} u + k(u) +
  f(u) \xi.
\]
 The rule $ \mathcal{R}_{gkpz} $ for building $ \mathcal{T}_{gkpz} $ is given by: 
\begin{equs}
 \mathcal{R}_{gkpz}(\Xi) & = 
 \lbrace () \rbrace, \\ \mathcal{R}_{gkpz}(\CI) & =  \lbrace (),
 ([\CJ]_{\ell}), ( [\CJ]_{\ell}, \CJ_1 ), ( [\CJ]_{\ell}, \CJ_1, \CJ_1 ), ([\CJ]_{\ell}, \Xi),
 \ell \in \N \rbrace,
\end{equs}
where  $ [\CJ]_{\ell} $ is a shorthand notation for $ \CJ,...,\CJ $ where $ \CJ $  is repeated $ \ell $ times. The admissible map 
$ R_{gkpz} $ associated to the generalised KPZ is defined by $ R_{gkpz} = ( \ell_{gkpz} \otimes \id) \Deltam_r $ where $ \ell_{gkpz} $ is non zero on trees with negative degree which do not contain any $ X $.
\begin{proposition} The map $ M =  M_{gkpz} $ satisfies only the property \ref{prop3}.
\end{proposition}

\begin{proof}
One counterexample, for the properties \ref{prop1} and \ref{prop2} is given by:
 $ \tau = \CI(\CI(\CI(\Xi) \Xi) \Xi) $, we have
\[
 \begin{aligned}
\Gamma_g M_{gkpz} \tau &= 
\Gamma_g \tau - C_1 \Gamma_g \CI(\CI(\Xi)) \\
M_{gkpz} \Gamma_g \tau &= \Gamma_g \tau 
- C_1 \CI( \Gamma_g \CI(\Xi) ),
 \end{aligned}
\]
where $ C_1 = \ell_{gkpz}(\CI(\Xi) \Xi)$.
Now $ \Gamma_g \CI(\CI(\Xi)) - \CI( \Gamma_g \CI(\Xi) )  $ is a polynomial different from zero. Similarly, one can check that $ \Pi_x^M \tau \neq \Pi_x M \tau $.
For the property \ref{prop3}, we use the Proposition~\ref{prop_PiM}. We  need to check that property on every $ \CI(\tau) \in \mathcal{T}_{gkpz} $ with negative degree. Such terms do not exist. Then  
\[ \lbrace \mathcal{I}_{1}(\tau) : \; \tau \in \mathcal{T}_{gkpz} \ \text{and} \ | \mathcal{I}_{1}(\tau) |_{\s} < 0  \rbrace  = \lbrace  \CI_1(\Xi) , \ \CI_1(\CI_1(\Xi)^2), \ \CI_1(\CI(\Xi)\Xi) \rbrace. \]
 and $ \Pi_x^{M} \CI_1(\tau)(x) = (\Pi_x  M \CI_1(\tau))(x) $ for $ \tau \in \lbrace \Xi, \ \CI_1(\Xi)^2, \ \CI(\Xi) \Xi \rbrace $.
\end{proof}
\subsection{The  stochastic quantization}
 The  stochastic quantization is given in dimension $ 3 $ by: 
\[
 \partial_{t} u = \Delta u + u^3 + \xi.
\]
and has been studied in \cite{reg}.
The rule $ \mathcal{R}_{qua} $ for building $ \mathcal{T}_{qua} $ is: 
\begin{equs}
 \mathcal{R}_{qua}(\Xi) = \lbrace () \rbrace, \quad \mathcal{R}_{qua}(\CI) = \lbrace (), (\Xi), (\CI), (\CI,\CI),   
 (\CI,\CI,\CI)  \rbrace.
\end{equs}
The admissible map 
$ R_{qua} $ associated to the stochastic quantization is defined by $ R_{qua} = ( \ell_{qua} \otimes \id) \Deltam_r $ where $ \ell_{qua} $ is non zero for:  
 \begin{equs}
\tau \in \lbrace  \CI( \Xi)^2,  \CI(\Xi)^2 \CI(   \CI(\Xi)^2 ) \rbrace.
  \end{equs}

\begin{proposition} The map $ M=M_{qua} $ satisfies only the property \ref{prop3}.
\end{proposition}

\begin{proof}
 For the properties \ref{prop1} and \ref{prop2}, a good counterexample is $ \tau = \CI(\CI(\Xi)^3) $.
For the property \ref{prop3}, 
$ \lbrace \mathcal{I}(\tau) : \; \tau \in \mathcal{T}_{qua} \; \text{and} \ | \mathcal{I}(\tau) |_{\s} < 0  \rbrace = \lbrace  \CI(\Xi)  \rbrace. $ Then it is obvious that
$ \Pi_x^M \CI(\Xi) = \Pi_x  M \CI(\Xi) $ .
\end{proof}

\appendix

\section{Alternative recursive proof}
    
    \begin{proof}[of Proposition~\ref{proposition_coassociativity_delta1}]
    We want to prove on $ \Tra_{\rho} $ that:
    \begin{equs}
     (\id \otimes \hat \Delta_1)\hat \Delta_1 = (\hat \Delta_1 \otimes \id)\hat \Delta_1.
    \end{equs}
    Since both maps are multiplicative for the product $ \star $ and the identity obviously holds when applied
    to $\one$, $X_i$ or $\Xi_{\mfl}$, it suffices to verify that it also holds for elements of the
    form $\CI^{\Labhom}_k(\tau)$. For this, note first that $ \hat \Delta_1$ has the following properties. 
    For $\sigma \tau \in \CF_{\rho}$ where $ \tau \in \Trees $ and $ \sigma = \prod_i \mathcal{C}(\tau_i) $, $ \tau_i \in \Trees $, one has by definition
    \begin{equs}
    \hat \Delta_1 \Co(\sigma \tau) & =  \DeltaQ \sigma \Co(\tau)\one = \DeltaQ(\sigma\Co(\tau)) \\ & 
    = (\DeltaQ \sigma)(\Co\otimes \Co)\DeltaQ \tau = (\Co\otimes \Co) \DeltaQ (\sigma\tau)\;.
    \end{equs}
    Furthermore, one has the identity
    \begin{equs}
    \DeltaQ \CI^{\Labhom}_k(\sigma \tau) &=  \DeltaQ \sigma \CI^{\Labhom}_k(\tau) = (\DeltaQ \sigma) \DeltaQ \CI^{\Labhom}_k(\tau) \\ & 
    = (\DeltaQ \sigma)(\CI^{\Labhom}_k \otimes \id + \sum_{\ell \in \N^{d+1}} \frac{X^{\ell}}{\ell !} \Co \otimes \CI^{\Labhom}_{k+ \ell})\DeltaQ \tau \\
    & = (\CI_k^{\Labhom} \otimes \id + \sum_{\ell \in \N^{d+1}} \frac{X^{\ell}}{\ell !} \Co \otimes \CI^{\Labhom}_{k+ \ell})(\DeltaQ \sigma \DeltaQ \tau)\\ &  = (
    \CI^{\Labhom}_k \otimes \id + \sum_{\ell \in \N^{d+1}} \frac{X^{\ell}}{\ell !} \Co \otimes \CI^{\Labhom}_{k+ \ell})( \DeltaQ \sigma\tau)\;.
    \end{equs}
    It follows that for any $\tau \in  \Trees$ one has the identity
    \begin{equs}
    ( \DeltaQ \otimes \id)\DeltaQ \CI^{\Labhom}_k(\tau)
    &= ( \DeltaQ \CI^{\Labhom}_k \otimes \id + \DeltaQ  \sum_{\ell \in \N^{d+1}} \frac{X^{\ell}}{\ell !}  \Co \otimes \CI^{\Labhom}_{k+\ell})\DeltaQ \tau \\
    &= ((\CI^{\Labhom}_k \otimes \id + \sum_{\ell \in \N^{d+1}} \frac{X^{\ell}}{\ell !} \Co \otimes \CI^{\Labhom}_{k+\ell}) \DeltaQ \otimes \id \\ &  + \sum_{\ell,m \in \N^{d+1}} ( \frac{X^{\ell}}{\ell !} \Co \otimes  \frac{X^{m}}{m !} \Co) \DeltaQ  \otimes \CI^{\Labhom}_{k + \ell + m })\DeltaQ \tau \\
    &= (\CI^{\Labhom}_k \otimes \id \otimes \id + \sum_{\ell \in \N^{d+1}} \frac{X^{\ell}}{\ell !} \Co \otimes \CI^{\Labhom}_{k+\ell} \otimes \id \\ & + \sum_{\ell,m \in \N^{d+1}}  \frac{X^{\ell}}{\ell !} \Co\otimes  \frac{X^{m}}{ m!} \Co \otimes \CI^{\Labhom}_{k + \ell + m})( \DeltaQ \otimes \id)\DeltaQ \tau\;.
    \end{equs}
   On the other hand, we have
    \begin{equs}
    (\id \otimes \DeltaQ)\DeltaQ \CI^{\Labhom}_k(\tau)
    & = (\CI^{\Labhom}_k \otimes \id \otimes \id  +\sum_{\ell \in \N^{d+1}} \frac{X^{\ell}}{\ell !} \Co \otimes \CI^{\Labhom}_{k+ \ell} \otimes \id \\ & + \sum_{\ell,m \in \N^{d+1}}  \frac{X^{\ell}}{\ell !} \Co\otimes  \frac{X^{m}}{m !} \Co \otimes \CI^{\Labhom}_{k + \ell + m})(\id \otimes \DeltaQ)\DeltaQ \tau\;,
    \end{equs}
    the claim follows by induction.
    \end{proof}

\endappendix

\bibliographystyle{./Martin}
\bibliography{refs_recursive_formulas}

\end{document}